\documentclass[11pt]{amsart}
\usepackage{a4wide}
\usepackage{hyperref}
\usepackage{bookmark}
\parskip=\baselineskip

\let\oldlist=\list
\newlength\oldparskip
\def\list#1#2{\oldparskip=\parskip\parskip=0\baselineskip
\oldlist{#1}{#2}\parskip=0\baselineskip}
\let\oldendlist=\enditemize
\def\enditemize{\oldendlist\vspace*{-\oldparskip}}

\renewcommand{\address}[1]{\thanks{#1}}

\usepackage{todonotes}

\usepackage{amsmath}
\usepackage{amssymb}
\usepackage{amsthm}
\usepackage{amscd}
\usepackage{tikz}
\usepgflibrary{arrows}
\newtheorem{theorem}{Theorem}[section]
\newtheorem{lemma}[theorem]{Lemma}
\newtheorem{remark}[theorem]{Remark}

\newtheorem{proposition}[theorem]{Proposition}
\newtheorem{definition}[theorem]{Definition}

\newcounter{step}[theorem]

\newcommand\cC{\mathcal{C}}
\newcommand\cT{\mathcal{T}}
\newcommand\pprime{\prime\prime}

\newcommand\norm[1]{\left\lVert #1\right\rVert}

\newcommand\abs[1]{\left\lvert #1\right\rvert}
\makeatletter
\def\inprod{\@ifnextchar_{\leftinprod}{\ninprod}}
\makeatother
\def\leftinprod_#1#2#3{\,\vphantom{\ninprod{#2}{#3}}_{#1}\!\!\ninprod{#2}{#3}}
\newcommand\ninprod[2]{\left\langle #1, #2\right\rangle}
\newcommand\hsnorm[1]{\norm{#1}_{\text{\tiny HS}}}
\newcommand\snorm[1]{\norm{#1}_{\text{\tiny S}}}
\DeclareMathOperator{\IC}{\mathbb{C}}
\DeclareMathOperator{\IR}{\mathbb{R}}

\DeclareMathOperator{\IQ}{\mathbb{Q}}
\DeclareMathOperator{\IZ}{\mathbb{Z}}
\DeclareMathOperator{\IN}{\mathbb{N}}

\DeclareMathOperator{\F}{F}

\DeclareMathOperator{\cF}{\mathcal{F}}

\DeclareMathOperator{\Perm}{S}
\def\MatrixGroup#1_#2#3{\mathchoice{#1_{#2}\hspace*{-0.3mm}#3}{#1_{#2}\hspace*{-0.4mm}#3}{#1_{#2}\hspace*{-0.4mm}#3\hspace*{0.3mm}}{#1_{#2}\hspace*{-0.4mm}#3\hspace*{0.3mm}}%
  \def\tempa{#3}%
  \def\tempb{\IQ}\ifx\tempa\tempb\else%
  \def\tempb{\IZ}\ifx\tempa\tempb\else%
  \def\tempb{R}\ifx\tempa\tempb\else%
  \def\tempb{(\F_2[X])}\ifx\tempa\tempb\else%
  \def\tempb{(\F_2)}\ifx\tempa\tempb\else%
  \def\tempb{\IC}\ifx\tempa\tempb\else%
  \fout\fi\fi\fi\fi\fi\fi%
}

\newcommand{\FG}{\mathbb{F}}
\DeclareMathOperator{\free}{\ast}
\DeclareMathOperator{\bigfree}{\ast}
\makeatletter
\def\quot{\@ifnextchar[{\sdp@quotleftA}{\sdp@quotright}}
\def\sdp@quotleftA[#1]#2{\@ifnextchar[{\sdp@quotboth[{#1}]{#2}}{\sdp@quotleft[{#1}]{#2}}}
\def\sdp@quotright#1[#2]{#1/#2}
\def\sdp@quotboth[#1]#2[#3]{#1\backslash #2 / #3}
\def\sdp@quotleft[#1]#2{#1\backslash #2}
\makeatother

\DeclareMathOperator{\Aut}{Aut}

\def\Autmp(#1,#2){\Aut_{#2}(#1)}
\def\Auttp(#1,#2){\Aut_{#2}(#1)}

\DeclareMathOperator{\Out}{Out}
\def\Outmp(#1,#2){\Out_{#2}(#1)}
\def\Outtp(#1,#2){\Out_{#2}(#1)}

\DeclareMathOperator{\op}{op}

\def\Rel(#1\actson #2){\mathcal{R}(#1\actson #2)}

\def\fullg(#1){\left[#1\right]}
\def\fullpg(#1){\left[\left[#1\right]\right]}

\DeclareMathOperator{\Bounded}{\mathcal{B}}
\DeclareMathOperator{\MatM}{M}

\DeclareMathOperator{\Lp}{L}
\DeclareMathOperator{\Lg}{L}

\newcommand{\bimod}[3]{\strut_{#1}#2\strut_{#3}}

\def\vnInterB(#1){\zH_{#1}}
\DeclareMathOperator{\Ad}{Ad}
\DeclareMathOperator{\actson}{\curvearrowright}
\DeclareMathOperator{\id}{id}

\newcommand\restrict[1]{\vert_{#1}}
\DeclareMathOperator{\Span}{span}

\def\op{\text{\scriptsize op}}
\def\bibounded#1{\overset{\circ}{#1}\vphantom{#1}}
\def\bbH{\bibounded{H}}

\def\pcT{\mathcal{T}^{\prime\prime}}
\def\Gauss{\Gamma}
\def\ppGauss{\Gauss^{\prime\prime}}
\makeatletter
\def\zH{\@ifnextchar_{\zH@}{\@undefined}}
\DeclareMathOperator{\zH@}{H}
\def\smallscripts#1{\@ifnextchar_{\smallscripts@{#1}}{\@undefined}}
\def\smallscripts@#1{\@ifnextchar_{\smallscripts@subP{#1}}{\@ifnextchar^{\smallscripts@superP{#1}}{\smallscripts@none{#1}}}}
\def\smallscripts@subP#1_#2{\@ifnextchar^{\smallscripts@subsuperP{#1}{#2}}{\smallscripts@sub{#1}{#2}}}
\def\smallscripts@subsuperP#1#2^#3{\smallscripts@both{#1}{#2}{#3}}
\def\smallscripts@superP#1^#2{\@ifnextchar_{\smallscripts@supersubP{#1}{#2}}{\smallscripts@super{#1}{#2}}}
\def\smallscripts@supersubP#1#2_#3{\smallscripts@both{#1}{#3}{#2}}
\def\smallscripts@none#1{{\displaystyle #1} }
\def\smallscripts@sub#1#2{{\displaystyle #1}_{\scriptscriptstyle #2}}
\def\smallscripts@super#1#2{{\displaystyle #1}^{\scriptscriptstyle #2}}
\def\smallscripts@both#1#2#3{{\displaystyle #1}_{\scriptscriptstyle #2}^{\scriptscriptstyle #3}}
\makeatother
\newcommand\chara{\smallscripts{\chi}}

\begin{document}
\title[Radial Multipliers on von Neumann algebras]{Radial multipliers on aribitrary amalgamated free products of finite von Neumann algebras}
\stepcounter{footnote}
\author{Steven Deprez}
\thanks{corresponding author, University of Copenhagen, sdeprez@math.ku.dk}
\thanks{Supported by ERC Advanced Grant no. OAFPG 247321}
\thanks{Supported by the Danish National Research Foundation through the Centre for Symmetry and Deformation}
\thanks{(DNRF92).} 
\address{Department of mathematics, Copenhagen university, Universitetsparken 5, 2500, Copenhagen O}

\begin{abstract}
  Let $(M_i)_{i}$ be a (finite or infinite) family of finite von Neumann algebras with a common subalgebra $P$. When $\varphi:\IN\rightarrow\IC$ is a function,
  we define the radial multiplier $M_\varphi$ on the amalgamated free product $M=M_1\free_P M_2\free_P\ldots$ setting $M_{\varphi}(x)=\varphi(n)x$ for every reduced expression
  $x$ of length $n$. In this paper we give a sufficient condition on $\varphi$ to ensure that the corresponding radial multiplier $M_\varphi$ is a completely bounded map, and
  moreover we give an upper bound on its completely bounded norm. Our condition on $\varphi$ does not depend on the choice of von Neumann algebras $(M_i)_i$ and $P$. This result extends
  earlier results by Haagerup and M\"oller, who proved the same statement for free products without amalgamation, and M\"oller showed that the same statement holds when $P$ has finite index
  in each of the $M_i$.

\end{abstract}

\maketitle

\section*{Introduction}
Let $\Gamma$ be a countable group. A function $f:\Gamma\rightarrow \IC$ is called a Herz-Schur multiplier
if the map $m_f:\IC\Gamma\rightarrow \IC\Gamma$ that is defined by $m_f(u_g)=f(g)u_g$ extends to a
$\sigma$-weakly continuous completely bounded map
$m_f:\Lg\Gamma\rightarrow\Lg\Gamma$. The Herz-Schur norm $\hsnorm{f}$ of $f$ is the completely-bounded norm of $m_f$. 

The free groups $\FG_n$ ($2\leq n\leq \infty$) come equiped with a natural word-length function $\abs{\cdot}$, associated to the canonical generators. Given a map $\varphi:\IN\rightarrow\IC$,
we can look at the associated radial function $f_\varphi:\FG_n\rightarrow \IC$, that is given by
$f_\varphi(g)=\varphi(\abs{g})$. Haagerup and Szwarc \cite{HSS10} showed that $f_\varphi$ is a Herz-Schur
multiplier if
and only if the H\"ankel matrix that is given by $H_\varphi=(\varphi(n+m)-\varphi(n+m+2))_{n,m}$ is trace
class. In that case,
there exist constants $c_\pm\in\IC$ such that $\varphi(n)=\varphi_0(n)+c_{+}+(-1)^nc_{-}$ where
$\varphi_0(n)$ converges to $0$.
The Herz-Schur norm of $f_\varphi$ is bounded by
\[\hsnorm{f_\varphi}\leq \norm{H_\varphi}_1+\abs{c_+}+\abs{c_-}.\]
In fact, this inequality becomes an equality if $n=\infty$, and Haagerup and Szwarc give an exact expression
for the Herz-Schur norm, also when $n$ is finite.

We say that a function $\varphi:\IN\rightarrow\IC$ is in class $\cC$ if the H\"ankel matrix $H_\varphi$ is
trace class, and we will write $\norm{\varphi}_{\cC}=\norm{H_\varphi}+\abs{c_+}+\abs{c_-}$.

Wysoczanski \cite{W95} studied radial multipliers on free products $\Gamma=\free_i\Gamma_i$ of groups.
A free product comes equiped with a natural word-length function $g\mapsto\abs{g}$, so it makes sense
to talk about radial multipliers on free products. Wysoczanski showed that a radial function $f_\varphi$
is a Herz-Schur multiplier if the H\"ankel matrix $K_\varphi=(\varphi(n+m)-\varphi(n+m+1))_{n,m}$
is trace class. If this is the case, we say that $\varphi$ is in class $\cC^\prime$.
If $\varphi\in\cC^\prime$, then the limit $c=\lim_{n\rightarrow \infty}\varphi(n)$ exists and the Hankel matrix
$\widetilde K_\varphi=(\varphi(n+m+1)-\varphi(n+m+2))_{n,m}$ is also trace class.
Wysoczanski showed that the Herz-Schur norm of $f_\varphi$ is bounded by
\[\hsnorm{f_\varphi}\leq \norm{\varphi}_{\cC^\prime}=
\norm{K_\varphi}_1+\norm{\widetilde K_\varphi}_1+\abs{c}.\]
When all the groups $\Gamma_i$ have the same (possibly infinite) cardinality, he can explicitly compute
the Herz-Schur norm of $f_\varphi$.

Wysoczanski's result has been extended to an operator algebraic framework by Haagerup and M\"oller
\cite{HM12}. Namely, when $M=\free_i M_i$ is a free product of von Neumann algebras, then $M$ is
$\sigma$-weakly densely spanned by the reduced words, i.e.\@ words of the form $x_1\ldots x_n$ where
$x_k\in M_{i_k}\setminus \IC1$ and $i_1\not=i_2\not=\ldots\not=i_n$. Given a function
$\varphi:\IN\rightarrow\IC$, we consider the radial multiplier $m_\varphi$ that is given by
$m_\varphi(x_1\ldots x_n)=\varphi(n)x_1\ldots x_n$ for every reduced word $x_1\ldots x_n$.
Haagerup and M\"oller showed that $m_\varphi$ is a completely bounded map if
the H\"ankel matrix $K_\varphi$ defined above is trace class. They also showed that in that case
$\norm{m_\varphi}_{cb}\leq\norm{\varphi}_{\cC^\prime}$. Recently, M\"oller
extended this result to include amalgamated free products where the amalgam has finite index in each of the
factors $M_i$ \cite{M13}.

The present paper extends this result to arbitrary amalgamated free products of finite von Neumann algebras.
We give an overview of the techniques used. It is instructive to first look at the ideas behind Wysoczanski's
result.

The Haagerup-Szwarc result can be restated as a result about radial Schur multipliers on trees, as we
explain now. In this form, the result was published by Haagerup, Steenstrup and Szwarc \cite{HSS10},
but is based on earlier unpublished notes by Haagerup and Szwarc.
Let $X$ be
a countable set, and consider a map $k:X\times X\rightarrow \IC$. The Schur product $k\star T$ of $k$ with a
bounded operator $T\in\Bounded(\ell^2(X))$ is given by
\[\inprod{\delta_x}{(k\star T)\delta_y}=k(x,y)\inprod{\delta_x}{T\delta_y}.\]
We say that $k$ is a Schur multiplier is $m_k:T\mapsto k\star T$ is a well-defined completely bounded
map from $\Bounded(\ell^2(X))$ to itself, and its Schur norm is $\snorm{k}=\norm{m_k}_{cb}$.
Suppose now that $d:X\times X\rightarrow \IR_+$ is a metric on $X$. We say that a Schur multiplier
$k:X\times X\rightarrow\IC$ is a \emph{radial} Schur multiplier if it is of the form $k(x,y)=\varphi(d(x,y))$
for some function $\varphi:\IR_+\rightarrow \IC$.

The concepts of Schur multipliers and Herz-Schur multipliers are closely related: a function
$f:\Gamma\rightarrow\IC$ on a group $\Gamma$ is a Herz-Schur multiplier if and only if the function
defined by $k(g,h)=f(g^{-1}h)$ is a Schur multiplier, and moreover $\hsnorm{f}=\snorm{k}$.

Let $T$ be a regular tree, which we identify with its vertex set. We consider $T$ as a metric space with the
usual metric $d:T\times T\rightarrow \IN$, given by the length of the shortest path between two vertices.
Haagerup, Steenstrup and Szwarc showed that
a radial function $k_\varphi:T\times T\rightarrow \IC$, defined by $k_\varphi(x,y)=\varphi(d(x,y))$
is a radial Schur multiplier if and only if $\varphi\in\cC$, and moreover its Schur norm is bounded by
$\snorm{k_\varphi}\leq \norm{\varphi}_{\cC}$.

Using this restatement, the proof of Wysoczanski's result goes as follows. Observe that
$\Gamma=\bigfree_i\Gamma_i$ is the fundamental group of the following graph of groups:

\hspace*{\stretch{1}}
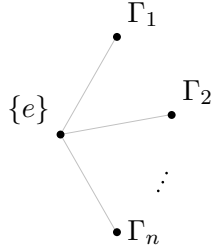
\begin{tikzpicture}[thick]
  \fill (0,0) circle(1.5pt) node[above left]{$\{e\}$};
  \fill (0,0) -- (60:1.5cm) circle(1.5pt) node[above right]{$\Gamma_1$};
  \fill (0,0) -- (10:1.5cm) circle(1.5pt) node[above right]{$\Gamma_2$};
  \fill (-20:1.5cm) circle(0.5pt) (-25:1.5cm) circle(0.5pt) (-30:1.5cm) circle(0.5pt);
  \fill (0,0) -- (-60:1.5cm) circle(1.5pt) node[right]{$\Gamma_n$};
\end{tikzpicture}\hspace*{\stretch{1}}

Consider the action of $\Gamma$ on the Bass-Serre tree $T$ associated to this graph of groups.
Observe that the vertices of the tree $T$ are indexed by
$\Gamma\sqcup \Gamma/\Gamma_1\sqcup\ldots\sqcup\Gamma/\Gamma_n$.
Consider the vertex $v_0\in T$ that corresponds to the identity in $\Gamma$.
For every group element $g\in\Gamma$, we see that the distance between $v_0$ and $g\cdot v_0$
is exactly twice the word-length of $g$. Let $\varphi:\IN\rightarrow \IC$ be a function.
Consider the function $\psi:\IN\rightarrow \IC$ that is defined
by $\psi(n)=\varphi(n/2)$ if $n$ is even, and $\psi(n)=0$ if $n$ is odd. The function $\psi$ is in class $\cC$
if and only if $\varphi$ is in class $\cC^\prime$, and moreover
$\norm{\psi}_{\cC}=\norm{\varphi}_{\cC^\prime}$.
Consider the radial functions $k=k_\psi:T\times T\rightarrow \IC$ and $f=f_\varphi:\Gamma\rightarrow \IC$
as before. We define an isometry $u:\ell^2(\Gamma)\rightarrow\ell^2(T)\otimes\ell^2(\Gamma)$ by the formula
$u(\delta_g)=\delta_{g^{-1}v_0}\otimes\delta_g^{-1}$.
Let $\Lg(\Gamma)$ act on $\ell^2(\Gamma)$ by left multiplication, i.e.\@
$u_g\delta_h=\delta_{gh}$.
We observe that the multipliers $m_f$ and $m_k$ are related by
$m_f(x)=u^\ast (m_k\otimes\id)(uxu^\ast) u$ for all $x\in \Lg(\Gamma)$.
In particular, we see that
\[\norm{m_f}_{cb}\leq\norm{m_k}_{cb}\leq\norm{\psi}_{\cC}=\norm{\varphi}_{\cC^\prime}.\]

At no point in the argument above did we use the fact that $\Gamma$ was a free product as opposed to an
amalgamated free product. We prove our result about amalgamated free products of von Neumann algebras
using a similar strategy. For this, we have to find a good replacement for the ``fundamental group of
a graph of groups''. We think that the right notion is that of the ``relative Gaussian construction'',
which we describe in short below, and in more detail in section \ref{sect:rg}. They come naturally with
a notion of a ``reduced word'', so it makes sense to talk about radial multipliers. We show that, for
an important subclass of the ``relative Gaussian constructions'', a radial multiplier $m_\varphi$ is
completely bounded, whenever $\varphi$ is in class $\cC$, and $\norm{m_\varphi}_{cb}\leq \norm{\varphi}_{\cC}$.
This is only true for a subclass of the ``relative Gaussian constructions'', as it fails for the classical
Gaussian construction.

The relative Gaussian construction is a strong generalization of Voiculescu's free Gaussian construction
\cite{V:FreeGauss}. Many generalizations of the free Gaussian construction have been introduced before.
For example, there are the $q$-Gaussian constructions \cite{BSp:qGauss}. The construction that is closest to our
``relative Gaussian construction'' is given by Shlyakhtenko's ``$A$-valued semicircular systems''
\cite{S99}. One can describe the
``relative Gaussian construction'' as being a tracial $A$-valued semicircular system,
but deformed in a way similar to
the $q$-Gaussian constructions.
As with the free Gaussian construction
and its generalizations, we first define a kind of \emph{Fock space} and creation and annihilation
operators on them. Then we define the ``relative Gaussian construction'' as the algebra generated by
certain creation and annihilation operators.

For the construction of the \emph{relative Fock space}, we need three pieces of data:
\begin{itemize}
\item A von Neumann algebra $M$ with a faithful normal state $\tau$.
\item A bimodule $H$ over $M$.
\item A self-adjoint $M$-$M$ bimodular contraction $F:H\otimes_M H\rightarrow H\otimes_M H$
  that satisfies the braid relation
  \[(F\otimes1_H)(1_H\otimes F)(F\otimes1_H)=(1_H\otimes F)(F\otimes1_H)(1_H\otimes F),\]
  as operators on $H\otimes_MH\otimes_M H$.
\end{itemize}

Consider the $n$-fold Connes tensor product bimodule $H^{(n)}$ of $H$ with itself. By convention,
we set $H^{(0)}=M$. We define a number of $M$-$M$ bimodular operators on $H^{(n)}$. Let $\sigma\in\Perm_n$
be a permutation on $n$ elements. It is well-known that $\sigma$ can be decomposed as $\sigma=t_1\ldots t_n$
where each $t_i=(k_i,k_i+1)$ is a transposition of two consecutive elements. Moreover, the shortest
such decomposition is unique, up to applying the braid relation
\[(k,k+1)(k,k-1)(k,k+1)=(k,k-1)(k,k+1)(k,k-1).\]
We assume that $\sigma=t_1\ldots t_n$ is such a shortest decomposition. Put
\[F_{(k,k+1)}=-1_{H^{(k-1)}}\otimes F\otimes 1_{H^{(n-k-1)}},\]
and observe that $F_\sigma=F_{t_1}\ldots F_{t_n}$ is well-defined. Define an operator $D_n$ by
\[D_n=\sum_{\sigma\in\Perm_n}F_\sigma.\]
Bozejko and Speicher showed in \cite{BSp:Coxeter} that such an operator $D_n$ is always a positive operator.
We define a new inner product $\inprod{\cdot}{\cdot}_F=\inprod{\cdot}{D_n\cdot}$ on $H^{(n)}$.
Denote by $H^{(n)}_F$ the $M$-$M$-bimodule that we obtain from $H^{(n)}$ by completion and separation
with respect to this new inner product $\inprod{\cdot}{\cdot}_F$. The relative Fock space is now
\[\cF_M(H,F)=\bigoplus_{n\geq 0}H^{(n)}_F\]

In $\cF_M(H,F)$, the vector $\hat 1\in \Lp^2(M)\subset \cF_M(H,F)$ plays a special role. It will be called the
\emph{vacuum vector} and is denoted by $\Omega$ to avoid confusion.
For every $\xi\in H$, we define a left creation operator $L(\xi)$ on $\cF_M(H,F)$, by the formula
\[L(\xi)\eta=\xi\otimes\eta\in H^{(n+1)}\subset H^{(n+1)}_F\text{ for all }\eta\in H^{(n)}.\]
This creation operator is in general a closable unbounded operator. We still denote its closure
by $L(\xi)$. The adjoint $L(\xi)^\ast$ is called an annihilation operator. When $F=0$,
then we get $D_n=1$ for all
$n\in\IN$, and the creation operators satisfy the relation
\[L(\xi)^\ast L(\eta)=\inprod{\xi}{\eta}_M.\]
In this case, we obtain the Fock space from Shlyakhtenko's $A$-valued semicircular elements.
If $M=\IC$ and $F(\xi\otimes\eta)=q\eta\otimes\xi$ for some $-1<q<1$, we obtain the Fock space of the
$(-q)$-Gaussian construction, and the creation and annihilation operators satisfy the relation
\[L(\xi)^\ast L(\eta)=\inprod{\xi}{\eta}+qL(\xi)L(\eta)^\ast.\]
In general, a similar relation holds, but it is slightly more complicated to write down. See section
\ref{sect:rg}, proposition \ref{prop:L}.
In any case, the algebra $\cT_M(H,F)=\Bounded_M(\cF_M(H,F))$
of bounded right-$M$-linear operators
on $\cF_M(H,F)$ is the $w^\ast$-closed linear span of the left action of $M$ and of
(the spectral projections of the self-adjoint part of) operators of the form
\[L(\xi_1)\ldots L(\xi_n)L(\eta_m)^\ast\ldots L(\eta_1)^\ast
\text{ where }\xi_1,\ldots,\xi_n,\eta_1,\ldots,\eta_m\in H.\]

For the relative Gaussian construction, we need one more piece of data:
an anti-unitary $J:H\rightarrow H$ that satisfies $J(x\xi y)=y^\ast J(\xi) x^\ast$,
and that is compatible with $F$ is some way. See section \ref{sect:rg} for more details.
Then we define $\Gamma^{\pprime}_{M}(H,F,J)$ to be the von Neumann algebra
on $\cF_{M}(H,F)$ that is generated by the left action of $M$ and all the
operators of the form $L(\xi)+L(J\xi)$ with $\xi\in H$.

Obviously the free Gaussian construction can be obtained from our construction by taking $M=\IC$, $F=0$ and
where $J$ is given by complex conjugation on $\ell^2(S)$ for a countable set $S$.
More generally, the $q$-Gaussian construction can be obtained by setting
$F(\xi\otimes\eta)=-q\eta\otimes\xi$. The tracial cases of
Shlyakhtenko's $A$-valued semicircular systems can be obtained by setting $F=0$ in our general construction.

For us, a more important class of examples is given by amalgamated free products and HNN extensions.
If $M_1\free_PM_2$ is an amalgamated free product, then we set $M=M_1\oplus M_2$, and
\[H=\Lp^2(M_1)\otimes_P\Lp^2(M_2)\oplus \Lp^2(M_2)\otimes_P\Lp^2(M_1)\]
The anti-unitary $J$ interchanges the two components of the direct sum above.
The operator $F$ is the projection onto the closed subspace
\begin{align*}
  &\Lp^2(M_1)\otimes_P \Lp^2(M_1)\oplus \Lp^2(M_2)\otimes_P\Lp^2(M_2)\\
  &\qquad=\Lp^2(M_1)\otimes_P \Lp^2(P)\otimes_P \Lp^2(M_1)\oplus \Lp^2(M_2)\otimes_P \Lp^2(P)\otimes_P \Lp^2(M_2)\\
  &\qquad\subset \Lp^2(M_1)\otimes_P \Lp^2(M_2)\otimes_P \Lp^2(M_1)
    \oplus \Lp^2(M_2)\otimes_P \Lp^2(M_1)\otimes_P \Lp^2(M_2)\\
  &\qquad=H^{(2)}.
\end{align*}
Then we obtain that
\[\Gamma^{\pprime}_M(H,F,J)=\MatM_2(\IC)\otimes (M_1\free_P M_2).\]

More generally, suppose we are given a graph $\Gamma$ of von Neumann algebras. A graph consists of a set $V$ of vertices,
and a set $E$ of edges, together with source and target maps $s,t:E\rightarrow V$, and with an involution
$\overline{\cdot}:E\rightarrow E$ that reverses each edge, so $s(\overline{e})=t(e)$, $t(\overline{e})=s(e)$
and $\overline{\overline{e}}=e$. A graph of von Neumann algebras is a graph $(V,E)$ together with a family of von Neumann
algebras $(M_v)_{v\in V}$ and a family of von Neumann subalgebras $P_e\subset M_{s(e)}$ ($e\in E$) and $\ast$-isomorphisms
$\alpha_e:P_e\rightarrow P_{\overline e}$ such that $\alpha_{\overline e}=\alpha_e^{-1}$. Moreover, we assume that there are
normal conditional expectations $E_e:M_{s(e)}\rightarrow P_e$.

Then we set $M=\bigoplus_{v\in V}M_v$ and we define $H$ to be the completion of
\[\bigoplus_{e\in E}\Lp^2(M_{s(e)})\otimes_{P_e}\Lp^2(M_{t(e)}),\]
where the inclusion of $P_e$ into $M_{t(e)}$ is given by the isomorphism
$\alpha_e:P_e\rightarrow P_{\overline e}\subset M_{t(e)}$. The Jordan subspace is spanned (as a real vector space) by elements of the form
\[x\otimes y+y^\ast\otimes x^\ast\in M_{s(e)}\otimes_{P_e}M_{t(e)}+M_{t(e)}\otimes_{P_{\overline{e}}}M_{s(e)}.\]
The operator $F$ is the projection onto the closed subspace of $H^{(2)}$ that is spanned by the spaces of the form
\[M_{t(e)}\otimes_{P_e} P_e\otimes_{P_e}M_{t(e)}.\]
We also call $\Gamma^{\pprime}_M(H,F,J)$ the fundamental von Neumann algebra of the given graph of von Neumann algebras.
With this terminology, we do not get that the group von Neumann algebra of a fundamental group of a gaph of groups is exactly the
fundamental von Neumann algebra of the corresponding graph of von Neumann algebras. Instead the fundamental von Neumann algebra is
the amplification of the group von Neumann algebra of the fundamental group. In fact, the fundamental von Neumann algebra is the
groupoid von Neumann algebra of the fundamental \emph{groupoid} of the graph of groups.

Observe that in both of the previous examples, the contraction $F$ is in fact a projection and moreover $1\otimes F$ commutes
with $F\otimes 1$ on $H^{(3)}$. We can not extend
Haagerup-M\"oller's result to arbitrary relative Gaussian constructions, but we can extend it to all the cases where the
contraction $F$ is a projection such that $1\otimes F$ commutes with $F\otimes 1$:
\begin{theorem}[see theorem \ref{thm:main}]
  \label{thm:main:intro}
  Let $(M,\tau)$ be a tracial von Neumann algebra and let $H$ be a Hilbert $M$-$M$ bimodule. Assume that $F$ is a projection onto
  an $M$-$M$ subbimodule of $H\otimes_M H$ such that $F\otimes1$ commutes with $1\otimes F$ on $H\otimes_MH\otimes_MH$.
  Let $\varphi:\IN\rightarrow\IC$ be a function in class $\cC$ defined above. Then there is a unique completely bounded map
  $\Phi_\varphi:\Bounded_M(\cF_M(H,F))\rightarrow\Bounded_M(\cF_M(H,F))$ that satisfies
  \begin{align*}
    \Phi_\varphi(L(\xi_1)\ldots L(\xi_n)L(\eta_m)^\ast\ldots L(\eta_1))&=\varphi(n+m)L(\xi_1)\ldots L(\xi_n)L(\eta_m)^\ast\ldots L(\eta_1)\\
    \Phi_\varphi(x)&=\varphi(0)x\qquad\text{ for all }x\in M\text{ acting on the left}\\
    \norm{\Phi_\varphi}_{cb}&\leq\norm{\varphi}_{\cC}.
  \end{align*}
\end{theorem}
The result about amalgamated free products follows by restricting this completely bounded map to a corner of a subspace.

\section{Preliminaries and Notation}
In this paper, the set of natural numbers $\IN$ includes $0$. When we want to refer to the natural numbers excluding $0$, we write $\IN_1$.

\subsection{Permutation groups}
\label{prelim:perm}
We consider the permutation group $\Perm_n$ to be the group of permutations of the set $\{1,\ldots,n\}$.
This permutation group $\Perm_n$ ($1\leq n<\infty$) is generated by the transpositions of consecutive
elements $t_1=(1,2),\ldots,t_{n-1}=(n-1,n)$. A complete set of relations for these generators is given by
\begin{align}
  t_i^2&=e&&\qquad\text{for all }1\leq i\leq n-1\\
  t_it_j&=t_jt_i&&\qquad\text{whenever }\abs{j-i}\geq 2\label{rel:cox2}\\
  t_it_{i+1}t_{i}&=t_{i+1}t_it_{i+1}&&\qquad\text{for all }1\leq i\leq n-2\label{rel:cox3}
\end{align}
The permutation group $\Perm_n$ is a Coxeter group with Coxeter system $(\Perm_n,\{t_i\})$, and most of the
results mentioned below are true in the more general context of Coxeter groups.

We write $\abs{\sigma}$ for the length of an element $\sigma\in S_n$ with respect
to the generating set $\{t_1,\ldots,t_{n-2}\}$. Observe that every element $\sigma\in\Perm_n$ is
represented by a word $t_{i_1}\ldots t_{i_s}$ of minimal length, and that this word is unique up to
relations (\ref{rel:cox2}) and (\ref{rel:cox3}). The length $\abs{\sigma}$ can also be computed as the number
of inversions of $\sigma$, i.e.\@ the number of pairs $1\leq i<j\leq n$ such that $\sigma(i)>\sigma(j)$.

Let $\pi$ be a partition of $\{1,ldots,n\}$ into consecutive sets, i.e. $\pi$ is of the form
\begin{equation}
\label{eqn:pi}
\pi=\left\{\{1,\ldots,k_1\},\{k_1+1,\ldots,k_1+k_2\},\ldots,\left\{1+\sum_{i=1}^{s-1}k_i, \ldots, \sum_{i=1}^{s}k_i\right\}\right\}
\end{equation}
for some $k_1,\ldots,k_s\geq 1$ with $\sum_ik_i=n$.
Then we can consider the subgroup $\Perm_{\pi}$ of permutations that preserve the partition $\pi$, i.e.
\[\Perm_\pi=\{\sigma\in \Perm_n\mid \sigma(X)=X\,\forall X\in\pi\}\cong \Perm_{k_1}\oplus\ldots\oplus \Perm_{k_s}.
\]
Such a subgroup is called a \emph{parabolic} subgroup of $\Perm_n$.
Every left coset $C$ of $\Perm_\pi$
contains a unique element $\sigma_C$ of minimal length. So every element $\sigma\in C$ can be written
uniquely as a product $\sigma=\sigma_C\sigma_0$ where $\sigma_0\in \Perm_\pi$. This decomposition
satisfies $\abs{\sigma}=\abs{\sigma_C}+\abs{\sigma_0}$.
The set of all element $\sigma_C$ is denoted by $V_\pi$.
This set $V_\pi$ can also be described directly in terms of the partition $\pi$:
\[V_\pi=\left\{\sigma\in \Perm_n\mid \sigma\restrict X\text{ is increasing for every }X\in \pi\right\}.\]
When $\pi$ is given explicitly in the form (\ref{eqn:pi}), we also write $V_{k_1,\ldots,k_s}=V_\pi$.
It is convenient to allow $k_i=0$ for some $i$. In that case, the corresponding set in $\pi$ is empty.

For $\sigma_1\in\Perm_n$ and $\sigma_2\in\Perm_m$, we write $\sigma_1\times\sigma_2$ for the permutation
in $\Perm_{n+m}$ that is given by
\[(\sigma_1\times\sigma_2)(i)=
\begin{cases}
  \sigma_1(i)&\text{if }i\leq n\\
  \sigma_2(i-n)+n&\text{if }i>n
\end{cases}
\]
The identity element in $\Perm_n$ is denoted by $\id_n$, or just $\id$ when $n$ is clear from the context.
We write $\sigma_{k,l}$ for the permutation in $\Perm_{k,l}$ that satisfies
\begin{equation}
  \label{eq:sigmanm}
  \sigma_{k,l}(i)=
\begin{cases}
  i+l&\text{if }i\leq k\\
  i-k&\text{if }i>k
\end{cases}
\end{equation}

\begin{lemma}
  \label{lem:perm}
  With this notation, we get the following relation between sets of the form $V_{k_1,\ldots,k_s}$:
  \begin{itemize}
  \item For every $n,m\in\IN_1$ and for all $k\geq -n$, we get
    \[V_{n+k,m}=\bigsqcup_{l=\max(-k,0)}^{\min(n,m)} (V_{n-l,l}\times V_{k+l,m-l})
    (\id_{n-l}\times \sigma_{k+l,l}\times \id_{m-l}).\]
  \end{itemize}
\end{lemma}
\begin{proof}
  The first two assertions follow immediately from the fact that every element of $\sigma\in V_{n,m,k}$ is the
  unique element of shortest length in the coset $\sigma(\Perm_n\times\Perm_m\times\Perm_k)$.

  The last assertion is most easily proven with the characterisation of $V_{n,m}$ in terms
  of the order on $\{1,\ldots,n+m\}$. The argument is best explained using the following picture:
  \\\hspace*{\stretch{1}}
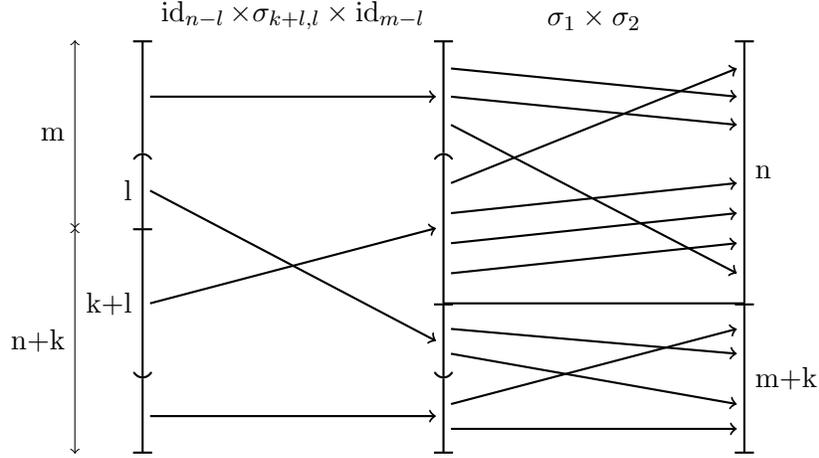
\begin{tikzpicture}[thick]
    \draw (2,5.5) node[above]{$\id_{n-l}\times\sigma_{k+l,l}\times\id_{m-l}$};
    \draw (6,5.5) node[above]{$\sigma_1\times\sigma_2$};
    \draw[|-] (0,0) -- (0,1);
    \draw[(-|] (0,1)-- (0,3) node[pos=0.5,left]{k+l};
    \draw[-)] (0,3) -- (0,4) node[pos=0.5,left]{l};
    \draw[-|] (0,4)-- (0,5.5);
    \draw[<->,thin] (-0.9,0) -- (-0.9,3) node[pos=0.5,left]{n+k};
    \draw[<->,thin] (-0.9,3) -- (-0.9,5.5) node[pos=0.5,left]{m};

    \draw[|-] (4,0) -- (4,1);
    \draw[(-|] (4,1)-- (4,2);
    \draw[-)] (4,2) -- (4,4);
    \draw[-|] (4,4)-- (4,5.5);

    \draw[|-|] (8,0) -- (8,2) node[pos=0.5,right]{m+k};
    \draw[-|] (8,2)-- (8,5.5) node[pos=0.5,right]{n};
    \draw[-] (4,2) -- (8,2);

    \draw[->] (0.1,0.5) -- (3.9,0.5);
    \draw[->] (0.1,2) -- (3.9,3);
    \draw[->] (0.1,3.5) -- (3.9,1.5);
    \draw[->] (0.1,4.75) -- (3.9,4.75);

    \draw[->] (4.1,0.33) -- (7.9,0.33);
    \draw[->] (4.1,0.66) -- (7.9,1.66);
    \draw[->] (4.1,1.33) -- (7.9,0.66);
    \draw[->] (4.1,1.66) -- (7.9,1.33);

    \draw[->] (4.1,2.4) -- (7.9,2.8);
    \draw[->] (4.1,2.8) -- (7.9,3.2);
    \draw[->] (4.1,3.2) -- (7.9,3.6);
    \draw[->] (4.1,3.6) -- (7.9,5.125);
    \draw[->] (4.1,4.375) -- (7.9,2.4);
    \draw[->] (4.1,4.75) -- (7.9,4.375);
    \draw[->] (4.1,5.125) -- (7.9,4.75);
  \end{tikzpicture}\hspace*{\stretch{1}}\\

  The elements $\sigma\in V_{n+k,m}$ are precisely the permutations that are increasing on each of
  $\{1,\ldots,n+k\}$ and $\{n+k+1,\ldots,n+k+m\}$ separately. Set $l$ to be the number of elements of
  $\{n+k+1,\ldots,n+k+m\}$ that are mapped into $\{1,\ldots,n\}$. These elements from necessarily an
  initial segment $\{n+k+1,\ldots,n+k+l\}$, because $\sigma$ is increasing on $\{n+k+1,\ldots,n+k+m\}$.
  Because $\sigma$ is a permutation, it maps the elements $\{n-l+1,\ldots,n+k\}$ into the set
  $\{n+1,\ldots,m+n+k\}$. It is now clear that $\sigma$ can be uniquely decomposed as
  $\sigma=(\id_{n-l}\times \sigma_{k+l,l}\times\id_{m-l})(\sigma_1\times\sigma_2)$ where
  $\sigma_1\in V_{n-l,l}$ and $\sigma_2\in V_{k+l,m-l}$. Conversely, it is easy to see that every such element
  is increasing on $\{1,\ldots,n+k\}$ and $\{n+k+1,\ldots,n+k+m\}$.
\end{proof}

\subsection{Bimodules}
Let $M$ be a von Neumann algebra. A left $M$-module is a Hilbert space $H$ with a normal representation
$\lambda:M\rightarrow\Bounded(H)$. Given a normal tracial state $\tau$ on $M$, we can perform the GNS
construction: on $M$ we consider the inner product given by $\inprod{x}{y}=\tau(x^\ast y)$. The completion
of $M$ with respect to this inner product is the GNS construction of $M$ and is written as $\Lp^2(M,\tau)$.
We also write $\norm{x}_2^2=\tau(x^\ast x)$.
The Hilbert space $\Lp^2(M,\tau)$ comes with a natural map $M\ni x\mapsto \hat x\in\Lp^2(M,\tau)$, and this map
has dense range. There is a natural representation of $M$ on $\Lp^2(M,\tau)$ that is given by
$\lambda(x)\hat y=\widehat{xy}$.
If it is clear from the context which trace we use, we just write
$\Lp^2(M)$. Left $M$-modules over von Neumann algebras are not
very interesting: every left $M$-module $H$ is isomorphic to $\bigoplus_i \Lp^2(M,\tau)p_i$
where the $p_i$ are projections in $M$ and $\tau$ is any faithful normal tracial state on $M$.

An $M$-$M$ bimodule is a Hilbert space $H$ with two normal representations $\lambda:M\rightarrow \Bounded(H)$
and $\rho:M^{\op}\rightarrow \Bounded(H)$ such that $\lambda(x)$ commutes with $\rho(y)$ for all $x\in M$ and
$y\in M^{\op}$. We will write $x\xi y=\lambda(x)\rho(y^{\op})\xi$ for all $x,y\in M$. The GNS construction
$\Lp^2(M,\tau)$ is an $M$-$M$-bimodule where $\rho$ is given by $\rho(x^\op)\hat y=\widehat{yx}$.

Let $M$ be a von Neumann algebra and fix a trace $\tau$ on $M$. Let $H$ be an $M$-$M$-bimodule and $\xi\in H$
a vector.
Then we can define an unbounded operator $l(\xi):\Lp^2(M)\rightarrow H$ by the formula $l(\xi)\hat x=\xi x$.
We say that a vector $\xi$ is left-bounded if $l(\xi)$ is a bounded operator. When $\xi,\eta\in H$ are
left-bounded vectors, then we can define an operator
$l(\xi)^\ast l(\eta):\Lp^2(M,\tau)\rightarrow\Lp^2(M,\tau)$. This operator commutes with
the right representation of $M$ on $\Lp^2(M,\tau)$, so it is of the form $\lambda(x)=l(\xi)^\ast l(\eta)$ for
some $x\in M$. We denote this $x\in M$ by $\inprod{\xi}{\eta}_M$.
The map $\inprod{\cdot}{\cdot}_M$ is called the (right) $M$-valued inner product on $H$. It is easy to see
that this
inner product satisfies $\inprod{x\xi y}{\eta z}_M = y^\ast\inprod{\xi}{x^\ast\eta}_M z$ for all $x,y,z\in M$,
and $\inprod{\xi}{\eta}_M^\ast = \inprod{\eta}{\xi}_M$. Moreover, the norm of $\xi$ is given by
$\norm{\xi}^2=\tau(\inprod{\xi}{\xi}_M)$.

We also consider the right multiplication operator $r(\xi)$ that is given by $r(\xi)\hat x=x\xi$.
A vector $\xi$ is said to be right-bounded if $r(\xi)$ is a bounded operator. For right-bounded operators
$\xi,\eta\in H$, we can also consider the operator $r(\xi)^\ast r(\eta)$ on $\Lp^2(M)$. This operator now
commutes with the left representation of $M$, so it is of the form $\rho(x^\op)^\ast$ for some $x\in M$.
This element $x$ is denoted by $\inprod_M{\xi}{\eta}$. The map
$\inprod_M{\cdot}{\cdot}$ is called the left $M$-valued inner product on $H$. It satisfies the relation
$\inprod_M{x\xi z}{y\eta}=x \inprod_M{\xi}{\eta z^\ast}y^\ast$.

We say that a vector is bi-bounded if it is both left and right-bounded. We denote the space of
all bi-bounded vectors in $H$ by $\bibounded{H}$. The space of all bi-bounded vectors is dense in $H$.

Let $H,K$ be two $M$-$M$ bimodules. The Connes tensor product $H\otimes_M K$ is the $M$-$M$ bimodule
defined as follows. On the algebraic tensor product
$\bibounded{H}\otimes^{alg}\bibounded{K}$, we define an inner product by the formula
\[\inprod{\xi\otimes\eta}{\zeta\otimes\theta}=\inprod{\eta}{\inprod{\xi}{\zeta}_M\theta}
=\inprod{\xi\inprod_M{\eta}{\theta}}{\zeta}
=\tau(\inprod_M{\eta}{\theta}^\ast\inprod{\xi}{\zeta}).
\]
The Connes tensor product $H\otimes_M K$ is the result of separation and completion of
$\bibounded{H}\otimes^{alg}\bibounded{K}$ with respect to this inner product.
Observe that when either $\xi$ is left-bounded or $\eta$ is right-bounded, the elementary tensor
$\xi\otimes\eta\in H\otimes_M K$ is well-defined. These elementary tensors satisfy the relation
$\xi x\otimes \eta=\xi \otimes x\eta$ for all $x\in M$.

\begin{lemma} Let $M$ be a finite von Neumann algebra and let $H,K$ be $M$-$M$ bimodules
  \begin{itemize}
  \item The set of left-bounded vectors in $H$ with the property that $\inprod{\xi}{\xi}_M$ is a
    projection in $M$, densely spans $H$.
  \item The set of elementary tensors $\xi\otimes\eta$ where $\inprod{\xi}{\xi}_M$ and
    $\inprod_M{\eta}{\eta}$ are the same projection in $M$, densely spans $H\otimes_M K$.
  \end{itemize}
\end{lemma}
\begin{proof}
  To prove the first point, we will show that every left-bounded vector $\xi\in H$ is the sum of two
  vectors $\xi_1,\xi_2$ with the property that $\inprod{\xi_i}{\xi_i}_M$ is a projection for $i=1,2$.
  Set $x=\sqrt{\inprod{\xi}{\xi}_M}$ and denote
  $\xi_n=\xi(\frac{n}{nx+1})$. Observe that $x_n:=\sqrt{\inprod{\xi_n}{\xi_n}_M}=\frac{nx}{nx+1}$
  increases to $p=\chara_{(0,\infty)}(x)$. In particular, it follows that $x_n$ converges to
  $p$ in $\norm{\cdot}_{2}$ and hence is a \hbox{$\norm{\cdot}_2$-Cauchy} sequence.
  But $\norm{\xi_n-\xi_m}=\norm{x_n-x_m}_2$ for all $n,m\in\IN$. So also the sequence $(\xi_n)_n$
  is a Cauchy sequence, and hence has a limit $\eta$. This limit satisfies
  $\inprod{\eta}{\eta}_M=p$ and moreover, $\xi=\eta x$. It suffices now to observe that $x$ is the linear
  combination $x=\frac12\norm{x}(u_+ + u_-)$ of the two unitaries
  $u_{\pm}=\frac{x}{\norm{x}}\pm i\sqrt{1-\frac{x^2}{\norm{x}^2}}$.

  For the second point, we will show that every elementary tensor product $\xi\otimes\eta$
  where $\xi$ is left-bounded and $\eta$ is right-bounded, can be written as a linear combination
  of two elementary tensors $\xi_i\otimes\eta_i$ ($i=1,2$) with the property that $\inprod{\xi_i}{\xi_i}_M$
  and $\inprod_M{\eta_i}{\eta_i}$ are the same projection ($i=1,2$). As in the proof of the first point,
  we find vectors $\zeta\in H$ and $\theta\in K$ and $x\in M$ such that
  $\xi\otimes\eta=\zeta\otimes x\theta$ and such that $p=\inprod{\zeta}{\zeta}_M$ and
  $q=\inprod_M{\theta}{\theta}$ are projections. We can assume that $p$ is the smallest projection in $M$
  that satisfies $px=x$ and similarly that $q$ is minimal with the property that $xq=x$. Write the
  polar decomposition of $x$ by $x=v\abs{x}$. Then we know that $vv^\ast=p$, $v^\ast v=q$ and $\abs{x}\in M$
  is a positive element. We write $\abs{x}=\frac12\norm{x}(u_+ + u_-)$ where
  $u_{\pm}=\frac{\abs{x}}{\norm{x}}\pm i\sqrt{1-\frac{\abs{x}^2}{\norm{x}^2}}$ as before. Observe that $u_\pm$ commutes
  with $\abs{x}$ and hence with its support $q$. In particular, we get that $v_\pm=vu_\pm$ satisfies
  $v_\pm v_\pm^\ast=p$ and $v_\pm^\ast v_\pm=q$. Now we see that the vectors $\xi_1=\xi_2=\zeta$, $\eta_1=v_+\theta$ and
  $\eta_2=v_-\theta$ satisfy our condition: $\xi\otimes\eta=\xi_1\otimes\eta_1 + \xi_2\otimes\eta_2$.
\end{proof}

\section{The relative Gaussian Construction}
\label{sect:rg}
In this section we define the relative Gaussian construction, giving formal definitions for all the concepts.
For a more accessible account of the construction, we refer to the introduction.

For the rest of this section, fix the following data.
\begin{itemize}
\item a finite von Neumann algebra $M$ with a trace $\tau$
\item an $M$-$M$ bimodule $H$
\item a self-adjoint $M$-$M$ bimodular contraction $F:H\otimes_M H\rightarrow H\otimes_M H$
that satisfies the braid relation
\[(F\otimes 1_{H})(1_H\otimes F)(F\otimes 1_H)=(1_H\otimes F)(F\otimes 1_H)(1_H\otimes F).\]
\end{itemize}

We denote the $n$-fold Connes tensor product by $H^{(n)}=H\otimes_M H \ldots \otimes_M H$.
By convention we write $H^{(0)}=\Lp^2(M)$ and $H^{(1)}=H$.  We will write
the identity operator of $H^{(n)}$ by $1_n$. On each $H^{(n)}$,
we define $M$-$M$ bimodular operators
$\left(F^{(n)}_\sigma\right)_{\sigma\in \Perm_n}$ in the following way.

We freely use the notations
introduced in subsection \ref{prelim:perm}.
For a transposition
$t=(i,i+1)$ of consecutive numbers, we set
\[F^{(n)}_t=-1_{i-1}\otimes F\otimes 1_{n-i-1}.\]
If $\sigma=t_{i_1}\ldots t_{i_k}$ is a decomposition of minimal length, then
the $M$-$M$ bimodular operator
\[F^{(n)}_\sigma=F^{(n)}_{t_{i_1}}\ldots F^{(n)}_{t_{i_k}}\]
does not depend on the choice of the (minimal) decomposition.
When $n,m\in\IN$, we also write $F_{n,m}=F_{\sigma_{n,m}}$ where $\sigma_{n,m}$ is as defined by
(\ref{eq:sigmanm}) in subsection \ref{prelim:perm}.

We set
\[D^{(n)}=\sum_{\sigma\in\Perm_n}F^{(n)}_\sigma.\]
By \cite[theorem 1.1]{BSp:Coxeter}, this operator $D^{(n)}$ is positive definite.
By convention, we set $D^{(0)}=1$ and $D^{(1)}=1_H$.


We define a new inner product on $H^{(n)}$ by the formula $\inprod{\xi}{\eta}_F=\inprod{\xi}{D^{(n)}\eta}$.
Apply separation/completion to this new inner product and denote the resulting $M$-$M$
bimodule by $H^{(n)}_F$.
Consider the natural map $I^{(n)}:H^{(n)}\rightarrow H^{(n)}_F$ and observe that this
is a bounded
$M$-$M$ bimodular operator that has dense range. We denote the identity operator on $H^{(n)}_F$ by $\id_n$.

\begin{lemma}
  For every $n,m\in\IN$, we get the following results
  \begin{itemize}
  \item The identity operator from on $H^{(n+m)}$ extends uniquely to a bounded operator
    \[I_{n,m}:H^{(n)}_F\otimes_M H^{(m)}_F\rightarrow H^{(n+m)}_F.\]
  \item The operator $F_{n,m}$ on $H^{(n+m)}$ extends uniquely to a contraction
    \[F_{n,m}:H^{(n)}_F\otimes_M H^{(m)}_F\rightarrow H^{(m)}_F\otimes_M H^{(n)}.\]
  \end{itemize}
\end{lemma}
\begin{proof}
  To prove the first point, observe that $D^{(n+m)}=E_{n,m}\left(D^{(n)}\otimes D^{(m)}\right)$. In particular,
  \[\left(D^{(n+m)}\right)^2=\left(D^{(n)}\otimes D^{(m)}\right)E_{n,m}^\ast E_{n,m}\left(D^{(n)}\otimes D^{(m)}\right)\leq \binom{n+m}{n}^2\left(D^{(n)}\otimes D^{(m)}\right)^2.\]
  Since the square root is an operator-monotone function, we also get that
  \[D^{(n+m)}\leq\binom{n+m}{n}(D^{(n)}\otimes D^{(m)})\]
  and hence that
  \[\norm{\xi}_{H^{(n+m)}_F}\leq\binom{n+m}{n}\norm{\xi}_{H^{(n)}_F\otimes_M H^{(m)}_F}.\]
  for all $\xi\in H^{(n+m)}$. So indeed, the identity operator on $H^{(n+m)}$ extends uniquely to a bounded operator
  \[I_{n,m}:H^{(n)}\otimes_M H^{(m)}\rightarrow H^{(n+m)}.\]
  The norm of this operator is less than $\binom{n+m}{n}$.


  As to the last point, remark that
  \[\sigma_{n,m}(\sigma_1\times\sigma_2)=(\sigma_2\times\sigma_1)\sigma_{n,m}\]
  for all permutations $\sigma_1\in\Perm_n$ and $\sigma_2\in\Perm_m$. Moreover, counting the
  number of inversions, it is easy to see that
  \[\abs{\sigma_{n,m}(\sigma_1\times\sigma_2)}=\abs{\sigma_{n,m}}+\abs{\sigma_1}+\abs{\sigma_2}=\abs{(\sigma_2\times\sigma_1)\sigma_{n,m}}.\]
  It follows that $F_{n,m}\left(D^{(n)}\otimes D^{(m)}\right)=\left(D^{(m)}\otimes D^{(n)}\right)F_{n,m}$.
  In particular, we see that $F_{n,m}^\ast F_{n,m}$ commutes with $D^{(n)}\otimes D^{(m)}$
  and hence with the square root $\left(D^{(n)}\otimes D^{(m)}\right)^{\frac12}$. So we see that
  \begin{align*}
    \norm{F_{n,m}\xi}_{H^{(m)}_F\otimes_MH^{(n)}_F}
    &=\inprod{\xi}{F_{n,m}^\ast \left(D^{(m)}\otimes D^{(n)}\right)F_{n,m}\xi}\\
    &=\inprod{\xi}{F_{n,m}^\ast F_{n,m} \left(D^{(n)}\otimes D^{(m)}\right)\xi}\\
    &=\inprod{\xi}{\left(D^{(m)}\otimes D^{(n)}\right)^{\frac12}F_{n,m}^\ast F_{n,m}\left(D^{(m)}\otimes D^{(n)}\right)^{\frac12}\xi}\\
    &\leq\inprod{\xi}{\left(D^{(n)}\otimes D^{(m)}\right)\xi}\\
    &=\norm{\xi}_{H^{(n)}_F\otimes_MH^{(m)}_F}.
  \end{align*}
  for all $\xi\in H^{(n+m)}_F$. So $F_{n,m}$ extends uniquely to a contraction
  \[F_{n,m}:H^{(n)}_F\otimes_M H^{(m)}_F\rightarrow H^{(m)}_F\otimes_M H^{(n)}_F.\]
\end{proof}

\begin{definition}
The relative Fock space of $M$, $H$ and $F$ is defined to be
\[\cF_M(H,F)=\bigoplus_i H_F^{(n)}.\]
We denote the algebraic direct sum by
\[\cF^0_M(H,F)=\bigoplus^{alg}_i H^{(n)}_F.\]
\end{definition}

Now we can define creation and annihilation operators. Let $T:H^{(n)}_F\rightarrow H^{(m)}_F$ be a bounded
right $M$-modular operator. Then we define a right-$M$-modular operator
$L_0(T):\cF^0_M(H,F)\rightarrow \cF_M^0(H,F)$ by the relation
that
\begin{align*}
  L_0(T)\eta&=I_{m,k}(T\otimes \id_{k})I_{n,k}^\ast\eta&&\text{ whenever }\eta\in H^{(n+k)}_F\\
  L_0(T)\eta&=0&&\text{ whenever }\eta\in H^{(k)}_F\text{ with }k<n.
\end{align*}
Each of the operators $I_{m,k}$ is bounded, but as $k\rightarrow\infty$, their norm may tend to $\infty$.
This is the case for the classical Gaussian construction. In such a case, the operators $L_0(T)$
need not be bounded. But the operators $L_0(T)$ are closable because they have a densely defined adjoint:
$L_0(T^\ast)\subset L_0(T)^\ast$.
\begin{definition} The creation operator on $\cF_M(H,F)$ associated to $T$ is the closure $L(T)$ of the
  densely defined operator $L_0(T)$.
\end{definition}
For every right-bounded vector in $\xi\in H$,
we get that the formula $l(\xi)\hat x=\xi x$ defines a bounded right-$M$-modular map $l(\xi):\Lp^2(M)\rightarrow H$.
The operator $L(l(\xi))$ is called the creation operator of $\xi$. We also denote this operator by $L(\xi)$.
The operator $L(l(\xi)^\ast)=L(\xi)^\ast$ is the corresponding annihilation operator.

\begin{proposition}
  \label{prop:L}
  The creation operators satisfy the following composition rule: for bounded right-$M$-modular operators
  $S:H^{(n_1)}_F\rightarrow H^{(m_1)}_F$ and $T:H^{(n_2)}_F\rightarrow H^{(m_2)}_F$, we get that
  \begin{equation}
    \label{eqn:prodL}
    L(S)L(T)=
    \sum_{k=0}^{\min(n_1,m_2)}L\left(
      \begin{aligned}
        &I_{m_1,m_2-k}(S\otimes \id_{m_2-k})(I_{k,n_1-k}\otimes\id_{m_2-k})\\
        &\qquad\qquad\qquad(\id_{k}\otimes F_{m_2-k,n_1-k})\\
        &\qquad(I_{k,m_2-k}^\ast\otimes\id_{n_1-k})(T\otimes\id_{n_1-k})I_{n_2,n_1-k}
      \end{aligned}\right).
  \end{equation}
  In particular, the space
  \[\cT_M(H,F)=\Span\{L(T)\mid T:H^{(n)}_F\rightarrow H^{(m)}_F\text{ for some }n,m\in\IN\}\]
  is a $\ast$-algebra whose commutant is the right action of $M$ on $\cF_M(H,F)$.
\end{proposition}
\begin{proof}
  It is clear from the definition of $I_{n,m}$ that, for all $n,m,l\in\IN$,
  \[I_{n+m,l}(I_{n,m}\otimes\id_l)=I_{n,m+l}(\id_n\otimes I_{m,l}),\]
  since both are the unique extension of the identity operator to a bounded linear map from $H^{(n)}_F\otimes_M H^{(m)}_F\otimes_M H^{(l)}_F$
  to $H^{(n+m+l)}_F$.

  Translating lemma \ref{lem:perm} in terms of the operators $I_{n,m}$ and $F_{n,m}$ yields the following relation
  \begin{align*}
    I_{n_1,l-n_1}^\ast I_{m_2,l-m_2}=\sum_{k=\max(0,n_1+m_2-l)}^{\min{n_1,m_2}}
    \begin{aligned}
      &(I_{k,n_1-k}\otimes I_{m_2-k,l+k-n_1-m_2})\\
      &(\id_k\otimes F_{m_2-k,n_1-k}\otimes \id_{l+k-n_1-m_2})\\
      &\qquad(I^\ast_{k,m_2-k}\otimes I^\ast_{n_1-k,l+k-n_1-m_2})
      \end{aligned}
  \end{align*}

  It is clear that both the left hand side and the right hand side of (\ref{eqn:prodL}) give $0$ when
  they are evaluated in an $\eta\in H^{(l)}_F$ with $l<n_2$ or $l<n_2+n_1-m_2$. Let $l\geq \max(n_1,m_2)$ and take a vector
  $\eta\in H^{(l-m_2+n_2)}$. Then we compute that
  \begin{align*}
    L(S)L(T)\eta&=I_{m_1,l-n_1}(S\otimes \id_{l-n_1})I_{n_1,l-n_1}^\ast I_{m_2,l-m_2}(T\otimes\id_{l-m_2})I_{n_2,l-m_2}^\ast\eta\\[2mm]
    &=\sum_{k=\max(0,n_1+m_2-l)}^{\min{n_1,m_2}}
    \begin{aligned}
      &I_{m_1,l-n_1}(S\otimes \id_{l-n_1})(I_{k,n_1-k}\otimes I_{m_2-k,l+k-n_1-m_2})\\
      &\qquad\qquad(\id_k\otimes F_{m_2-k,n_1-k}\otimes \id_{l+k-n_1-m_2})\\
      &\qquad(I^\ast_{k,m_2-k}\otimes I^\ast_{n_1-k,l+k-n_1-m_2})(T\otimes\id_{l-m_2})I_{n_2,l-m_2}^\ast\eta
    \end{aligned}\\[2mm]
    &=\sum_{k=\max(0,n_1+m_2-l)}^{\min{n_1,m_2}}
    \begin{aligned}
      &I_{m_1+m_2-k,l+k-n_1-m_2}(I_{m_1,m_2-k}\otimes\id_{l+k-n_1-m_2})\\
      &\qquad\qquad(S\otimes \id_{l-n_1})(I_{k,n_1-k}\otimes \id_{l-n_1})\\
      &\qquad\qquad(\id_k\otimes F_{m_2-k,n_1-k}\otimes \id_{l+k-n_1-m_2})\\
      &\qquad\qquad(I^\ast_{k,m_2-k}\otimes \id_{l-m_2})(T\otimes\id_{l-m_2})\\
      &\qquad(I_{n_2,n_1-k}^\ast\otimes\id_{l+k-n_1-m_2})I_{n_1+n_2-k,l+k-m_2-n_1}^\ast\eta
    \end{aligned}\\[2mm]
    &=\sum_{k=0}^{\min(n_1,m_2)}
    L\left(\begin{aligned}
        &I_{m_1,m_2-k}(S\otimes \id_{m_2-k})(I_{k}{n_1-k}\otimes\id_{m_2-k})\\
        &\qquad\qquad(\id_{k}\otimes F_{m_2-k,n_1-k})\\
        &\qquad(I_{k,m_2-k}^\ast\otimes\id_{n_1-k})(T\otimes\id_{n_1-k})I_{n_2,n_1-k}
      \end{aligned}\right)\eta.
  \end{align*}
  as claimed.

  It remains to show that the commutant of $\cT_M(H,F)$ is just the right action of $M$. It is clear that right multiplication by $M$ commutes with every operator in $\cT_M(H,F)$.
  Suppose that $T\in\Bounded(\cF_M(H,F))$ satisfies $TL(S)\subset L(S)T$ for all bounded right-$M$-modular operators $S:H^{(n)}_F\rightarrow H^{(m)}_F$. Let $\xi\in H^{(k)}_F$
  be a left-bounded vector. Then we see that
  \[T\xi=TL(l(\xi))\Omega=L(l(\xi))T\Omega,\]
  so $T$ is completely determined by its value $T\Omega$. On the other hand,
  \[\inprod{\xi}{T\Omega}=\inprod{\Omega}{L(l(\xi)^\ast)T\Omega}=\inprod{\Omega}{TL(l(\xi)^\ast)\Omega}=0\]
  for every $\xi\in H^{(k)_F}$ with $k>0$. Hence, $T\Omega\in\Lp^2(M)$. Since $T$ is a bounded
  operator, it is clear that $T\Omega=\hat x$ for some $x\in M$, so $T$ is given by right multiplication by $x$.
\end{proof}

We write $\pcT_M(H,F)$ for the bicommutant of $\cT_M(H,F)$, i.e.\@ all the bounded linear operators that commute with the right action of $M$.
As a von Neumann algebra, this is not a very interesting object, since $\pcT_M(H,F)\cong \Bounded(\ell^2(\IN))\otimes M$, but later on we will consider more interesting subalgebras.

Consider the orthogonal projection $P:\cF_M(H,F)\rightarrow\Lp^2(M)$, and observe that $PTP^\ast\in\Bounded(\Lp^2(M))$ commutes with the right action
of $M$, for all $T\in\pcT_M(H,F)$. So we can define a map $E:\pcT_M(H,F)\rightarrow M$ by the formula $E(T)=PTP^\ast$.
\begin{lemma}
  The map $E:\pcT_M(H,F)\rightarrow M$ is a normal conditional expectation, but not faithful.
\end{lemma}
\begin{proof}
  It is clear that $E$ is a normal unital completely positive map. The map $E$ is also $M$-$M$ bimodular because $P$ intertwines the left actions of $M$.
  To show that $E$ is not faithful, take any nonzero bounded right-$M$ linear map $T:H^{(n)}_F\rightarrow \Lp^2(M)$ with $n\not=0$. Then we see that $L(T^\ast)PL(T)\in\pcT_M(H,F)$ is a non-zero positive
  bounded operator with \[E(L(T^\ast)PL(T))=0.\]
\end{proof}

The relative Gaussian construction will be a special kind of subalgebra of $\pcT_M(H,F)$ on which $E$ is faithful.
In order to define the relative Gaussian construction, we need one more piece of data, namely an anti-unitary operator $J:H\rightarrow H$ that satisfies the following relations
\begin{itemize}
\item $J$ is an involution, i.e.\@ $J^2=\id_1$.
\item $J$ intertwines the left and right representations of $M$ on $H$, i.e.
  \[J(x\xi y)=y^\ast J(\xi) x^\ast\text{ for all }x,y\in M\text{ and }\xi\in H.\]
\item $J$ is compatible with $F$ in the sense that
  \begin{equation}
    \label{eqn:compat}
    (l(\zeta)^\ast\otimes \id_1) F (J\eta\otimes \xi)=J (l(\xi)^\ast)\otimes 1)F (\eta\otimes \zeta),
  \end{equation}
  for all bi-bounded vectors $\xi,\zeta\in H$ and all vectors $\eta\in H$.
\end{itemize}

Then we want to define the algebraic relative Gaussian construction $\Gauss_M(H,F,J)$ to be the subalgebra of $\cT_M(H,F)$ that is generated
by the left action of $M$ and by elements of the form $W(\xi)=L(\xi)+L(J\xi)^\ast$ where $\xi$ is a bi-bounded vector in $H$.
The relative Gaussian construction $\ppGauss_M(H,F,J)$ will then be the von Neumann algebra generated by $\Gauss_M(H,F,J)$. Observe that $J\xi$
is a left-bounded vector if and only if $\xi$ is right-bounded. More precisely, $l(J\xi)\hat x=Jr(\xi)x^\ast$. In particular, we get that $\inprod{J\xi}{J\eta}_M=\inprod_M{\xi}{\eta}$.
We consider the elements of $M$ to
be reduced words of length $0$ and elements of the form $W(\xi)$ are interpreted as reduced words of length $1$. These reduced words satisfy
$x\Omega=\hat x\in \Lp^2(M)\subset \cF_M(H,F)$ and $W(\xi)\Omega=\xi$. We will define longer ``reduced words'' $W(\xi)$ for vectors $\xi\in H^{(n)}_F$ by a generalization of
the Wick formula. The operator $W(\xi)$ will be the unique element in $\Gauss_M(H,F,J)$ that satisfies $W(\xi)\Omega=\xi$.

We extend $J$ to an anti-unitary operator $\widetilde J$ on $\cF_M(H,F)$. First, define an anti-unitary operator $J^{(n)}$ on $H^{(n)}$
by the relation that $J^{(n)}(\xi_1\otimes\ldots\otimes\xi_n)=J\xi_n\otimes\ldots\otimes J\xi_1$. Observe that this does indeed define an anti-unitary because by induction we get
\begin{align*}
  &\inprod{J^{(n)}(\xi_1\otimes\ldots\otimes \xi_n)}{J^{(n)}(\eta_1\otimes\ldots\otimes\eta_n)}\\
  &\qquad=\inprod{J^{(n-1)}(\xi_1\otimes\ldots\otimes\xi_{n-1})}{\inprod{J\xi_n}{J\eta_n}_MJ^{(n-1)}(\eta_1\otimes\ldots\otimes\eta_{n-1})}\\
  &\qquad=\inprod{J^{(n-1)}(\xi_1\otimes\ldots\otimes\xi_{n-1})}{J^{(n-1)}(\eta_1\otimes\ldots\otimes\eta_{n-1}\inprod_M{\eta_n}{\xi_n})}\\
  &\qquad=\inprod{(\eta_1\otimes\ldots\otimes\eta_{n-1})\inprod_M{\eta_n}{\xi_n}}{\xi_1\otimes\ldots\otimes\xi_{n-1}}\\
  &\qquad=\inprod{\eta_1\otimes\ldots\otimes\eta_n}{\xi_1\otimes\ldots\otimes\xi_n}
\end{align*}
By convention, we set $J^{(1)}=J$ and $J^{(0)}\hat x=\widehat{x^\ast}$.
The extended operator
\[\widetilde J:\cF_M(H,F)\rightarrow\cF_M(H,F)\]
is now the direct sum of the operators $J^{(n)}$. 

\begin{lemma}
  We get the following relation between $J^{(n)}$ and $F_\sigma^{(n)}$ for $\sigma\in\Perm_n$:
  \begin{equation}\label{eqn:JF}J^{(n)}F_\sigma^{(n)}J^{(n)}=F_{\gamma\sigma\gamma^{-1}},\end{equation}
  where $\gamma\in\Perm_n$ is defined by $\gamma(i)=n-i+1$.
\end{lemma}
\begin{proof}
  Once we prove that $J^{(2)}$ commutes with $F$, it is clear that $J^{(n)}(\id_{k-1}\otimes F\id_{n-k-1})J^{(n)}=\id_{n-k-1}\otimes F\otimes \id_{k-1}$
  for all $1\leq k\leq n-1$. So (\ref{eqn:JF}) holds for the generators of $\Perm_n$, and hence
  for all of $\Perm_n$. The fact that $J^{(2)}$ commutes with $F$ follows from our compatibility relation between $F$ and $J$:
  observe that (\ref{eqn:JF}) is equivalent to the fact that, for all bi-bounded vectors $\xi_1,\xi_2,\eta_1,\eta_2$,
  we have that
  \[\inprod{\xi_1\otimes\xi_2}{F(J\eta_1\otimes\eta_2)}=\inprod{F(\eta_1\otimes\xi_1)}{\eta_2\otimes J\xi_2}.\]
  Now we see that
  \begin{align*}
    \inprod{\xi_1\otimes\xi_2}{J^{(2)}F(\eta_1\otimes\eta_2)}
    &=\inprod{F(\eta_1\otimes\eta_2)}{J\xi_2\otimes J\xi_1}\\
    &=\inprod{\eta_2\otimes\xi_1}{F(J\eta_1\otimes J\xi_2)}\\
    &=\inprod{F(\eta_2\otimes\xi_1)}{J\eta_1\otimes J\xi_2}\\
    &=\inprod{\xi_1\otimes\xi_2}{FJ^{(2)}(\eta_1\otimes\eta_2)}
  \end{align*}
\end{proof}

\begin{definition}
Let $n,m\in\IN$ and let $\xi\in H^{(n+m)}_F$. We define an unbounded operator $S_{n,m}(\xi):H^{(n)}_F\rightarrow H^{(m)}_F$
by setting $S_{n,m}(\xi)\eta=J^{(m)}(l(\eta)^\ast\otimes\id_m)I_{n,m}^\ast J^{(n+m)}\xi$ for all left-bounded vectors $\eta\in H^{(n)}$.

We say that a vector $\xi\in H^{(n)}_F$ is \emph{bounded} if $S_{k,n-k}(\xi)$ is a bounded operator for all integers $k=0,\ldots,n$.
The space of all bounded vectors in $H^{(n)}_F$ is denoted by $\bbH^{(n)}_F$.
\end{definition}
When $\xi\in H=H^{(1)}_F$, then we see that $S_{0,1}(\xi)=l(\xi)$ while $S_{1,0}(\xi)=l(J\xi)^\ast$.
In particular, a vector $\xi\in H=H^{(1)}_F$ is bounded if and only if it is bi-bounded, and a vector $\xi\in\Lp^2(M)=H^{(0)}_F$
is bounded if and only if $\xi=\hat x$ for some $x\in M$.
Observe that $S_{n,m}(x\xi y+\eta)=xS_{n,m}(\xi)y+S_{n,m}(\eta)$.
We can also describe $S_{n,m}(\xi)$ by the relation that
\[\inprod{\eta}{S_{n,m}(\xi)\zeta}=\inprod{I_{m,n}\left(\eta\otimes J^{(n)}\zeta\right)}{\xi}.\]

In the special case where $F=0$, we see that
\[L(S_{n,m}(\xi_1\otimes\ldots\otimes\xi_n\otimes\eta_1\otimes\ldots\otimes\eta_m))=L(\xi_1)\ldots L(\xi_n)L(J\eta_1)^\ast\ldots L(J\eta_m)^\ast,\]
for all sets of bi-bounded vectors $\xi_1,\ldots,\xi_n,\eta_1,\ldots,\eta_m\in H$.

\begin{definition}
  Let $\xi\in H^{(n)}_F$ be a bounded vector. Then we define the word of $\xi$ to be
  \begin{equation}
    \label{eqn:Wick}
    W(\xi)=\sum_{k=0}^n L(S_{k,n-k}(\xi))
  \end{equation}
\end{definition}
Observe that $W(\xi)\Omega=\xi$.
For a bi-bounded vector $\xi\in H$, the formula above reads $W(\xi)=L(\xi)+L(J\xi)^\ast$.
When $F=0$, the formula above reduces to the usual Wick formula
\[W(\xi_1\otimes\ldots\otimes \xi_n)=\sum_{k=0}^nL(\xi_1)\ldots L(\xi_k)L(J\xi_{k+1})^\ast\ldots L(J\xi_n)^{\ast}.\]
For this reason, we still call (\ref{eqn:Wick}) the Wick formula for the relative Gaussian construction.

We define bilinear maps $\boxtimes_k:\bbH^{(n+k)}_F\times\, \bbH^{(k+m)}_F\rightarrow H^{(n+m)}_F$
by the formula
\[\xi\boxtimes_k\eta=I_{n,m}(S_{k,n}(\xi)\otimes\id_m)I^\ast_{k,m}\eta,\]
or equivalently,
\[\inprod{\zeta\otimes\theta}{\xi\boxtimes_k\eta}=\inprod{I^{\ast}_{n,k}J^{(n+k)}\xi\otimes\theta}{J^{(n)}\theta\otimes I^{\ast}_{k,m}\eta}\]
for all $\zeta\in H^{(n)}_F$ and $\theta\in H^{(m)}_F$.
Observe that $J^{(n+m)}(\xi\boxtimes\eta)=J^{(n+k)}\eta\boxtimes J^{(m+k)}\xi$ for all bounded vectors $\xi\in \bbH^{(n+k)}_F,\eta\in\bbH^{(m+k)}_F$.

In the case where $F=0$,
these bilinear maps can also be described by
\[(\xi_1\otimes\xi_2)\boxtimes_k(\eta_1\otimes\eta_2)=\xi_1\otimes\inprod{J^{(k)}\xi_2}{\eta_1}_M\eta_2,\]
for all $\xi_1\in \bbH^{(n)}$, $\eta_1,\xi_2\in \bbH^{(k)}$ and $\eta\in\bbH^{(m)}$.

\begin{lemma}
  Let $\xi\in H^{(n)}_F$ be a bounded vector. Then $J^{(n)}\xi$ is still a bounded vector and
  \[W(\xi)^\ast=W(S\xi).\]

  Let $\xi\in H^{(n)}_F$ and $\eta\in H^{(m)}_F$ be bounded vectors. Then the vectors $\xi\boxtimes_k \eta$ are
  still bounded vectors, for all $k=0,\ldots,\min(n,m)$. Moreover, we get the following product formula
  \[W(\xi)W(\eta)=\sum_{k=0}^{\min(n,m)}W(\xi\boxtimes_k\eta).\]

  In particular, the space
  \[\Gauss_M(H,F,J)=\Span\{W(\xi)\mid \xi\in H^{(n)}_F\text{ is bounded, }n\in\IN\}\]
  is a $\ast$-algebra. We call this $\ast$-algebra the algebraic relative Gaussian construction.
\end{lemma}
\begin{proof}
 In order to prove this theorem, it will be convenient to use the following variants of $S_{n,m}$ and $\boxtimes_k$.
 Let $K,L$ be $M$-$M$ bimodules. For a vector $\xi\in K\otimes_M \overline{L}$, we define an operator $T_{L,K}(\xi):L\rightarrow K$
 by the formula
 \[\inprod{\eta}{T_{L,K}(\xi)\zeta}=\inprod{\eta\otimes \overline{\zeta}}{\xi},\]
 for all $\eta\in K$ and for all left-bounded vectors $\zeta\in L$. This is not necessarily
 a bounded operator, but it is a closable, densely define unbounded operator. Indeed, $T_{L,K}$ is
 closable because its adjoint is given by
 \[T_{L,K}(\xi)^\ast\supset T_{K,L}(\overline{\xi}),\]
 where we identified $\overline{K\otimes_M\overline{L}}$ with $L\otimes_M\overline{K}$.
 We denote the closure of $T_{L,K}$ still by $T_{L,K}$.
 We say that $\xi$ is bounded for the tensor product decomposition $K\otimes_M \overline{L}$ if this operator $T_{L,K}(\xi)$ is bounded.

 Let $K, L_1, L_2$ be $M$-$M$ bimodules.
 For vectors $\xi\in L_1\otimes_M \overline{K}$ and $\eta\in K\otimes_M L_2$ that are bounded in these tensor product decompositions,
 we write $\xi\boxdot_K\eta=(T_{K,L_1}\otimes \id_{L_2})\eta$, or equivalently, $\xi\boxdot_K\eta$ is the unique vector in $L_1\otimes_M L_2$
 that satisfies
 \[\inprod{\zeta_1\otimes\zeta_2}{\xi\boxdot_K\eta}=\inprod{\overline{\xi}\otimes\zeta_2}{\overline{\zeta_1}\otimes\eta}.\]

 Observe that the anti-unitary operator $J:H\rightarrow H$ can also be interpreted as an $M$-$M$ bimodule
 isomorphism $J:H\rightarrow \overline{H}$. This is also true for the operators $J^{(n)}$ on $H^{(n)}_F$.
 The original $S_{n,m}$ and $\boxtimes_k$ can easily be expressed in terms of $T_{L,K}$ and $\boxdot_K$. 
 \begin{align*}
   S_{n,m}(\xi)&=T_{H^{(n)}_F,H^{(m)}_F}((\id_m\otimes J^{(n)})I^{\ast}_{m,n}\xi)&&\text{ for }\xi\in H^{(n+m)}_F\\
   \xi\boxtimes_k\eta&= I_{n,m}(((\id\otimes J^{(k)})I^{\ast}_{n,k}\xi)\boxdot(I^{\ast}_{k,m}\eta))&&\text{ for }\xi\in H^{(n+k)}_F\text{ and }\eta\in H^{(k+m)}_F.
 \end{align*}
 Moreover, a vector $\xi\in H^{(n)}_F$ is bounded if and only if $(\id_k\otimes J^{(n-k)})E^{\ast}_{n,n-k}\xi$ is bounded for the tensor product decomposition
 $H^{(n)}_F\otimes_M H^{(n-k)}_F$, for all $k=0,\ldots,n$.
 
 It is now clear that, whenever $\xi\in H^{(n)}\xi$ is a bounded vector, then $J^{(n)}\xi$ is still a bounded vector, and
 $S_{k,n-k}(\xi)^\ast=S_{n-k,k}(J^{(n)}\xi)$. It follows that $W(J^{(n)}\xi)=W(\xi)^\ast$.

 We prove the product formula in three steps.

 \textbf{step 1:} For all $n,m$ and for all bi-bounded vectors $\xi,\theta\in H^{(m)}$ and all vectors $\eta\in H^{(n)}$,
 we get
 \begin{equation}
   \label{eqn:FJmn}
   J^{(n)}(l(\xi)^\ast\otimes \id_n)F_{n,m}(\eta\otimes\zeta)=(l(\zeta)^\ast\otimes\id_n)F_{n,m}(J^{(n)}\eta\otimes\xi).
 \end{equation}
 Remark that (\ref{eqn:FJmn}) can be rewritten to the relation that
 \[\inprod{F_{n,m}(\eta\otimes\zeta)}{\xi\otimes J^{(n)}\theta}=\inprod{\zeta\otimes\theta}{F_{n,m}(J^{(n)}\eta\otimes\xi)},\]
 for all bi-bounded vectors $\xi,\zeta\in H^{(m)}_F$ and all vectors $\eta,\theta\in H^{(n)}_F$. Observe that,
 for fixed bi-bounded vectors $\xi,\zeta$, the left and right hand sides of the equation are continuous in $\eta,\theta$,
 so we only have to check this for dense sets of $\eta,\theta\in H^{(n)}_F$. 

 First we observe that the relation (\ref{eqn:FJmn}) is symmetric in $n,m$: if $F_{n,m}$ satisfies (\ref{eqn:FJmn}), then we see that,
 for all bi-bounded vectors $\xi,\zeta\in H^{(n)}_F$ and $\eta,\theta\in H^{(m)}_F$
 \begin{align*}
   \inprod{F_{m,n}(\eta\otimes\zeta)}{\xi\otimes J^{(m)}\theta}
   &=\inprod{\eta\otimes\zeta}{F_{n,m}(\xi\otimes J^{(m)}\theta)}\\
   &=\inprod{F_{n,m}(J^{(n)}\xi\otimes\eta)}{J^{(m)}\theta\otimes J^{(n)}\zeta}\\
   &=\inprod{\zeta\otimes\theta}{F_{m,n}(J^{(m)}\eta\otimes\xi)},
 \end{align*}
 where we used the fact that $F_{n,m}^\ast=F_{m,n}=J^{(n+m)}F_{n,m}J^{(n+m)}$. So also $F_{m,n}$ satisfies (\ref{eqn:FJmn}).
 
 We will prove (\ref{eqn:FJmn}) by induction on $m$ and $n$.
 So it suffices to show that $F_{n+k,m}$ satisfies (\ref{eqn:FJmn}), whenever $F_{n,m}$ and $F_{k,m}$ satisfy (\ref{eqn:FJmn}).
 Note that
 \[F_{n+k,m}=(\id_m\otimes I_{n,k})(F_{n,m}\otimes \id_k)(\id_n\otimes F_{k,m})(I_{n,k}^\ast\otimes\id_m).\]
 Let $\xi,\zeta\in H^{(m)}_F$ be bi-bounded vectors and let $\eta_1,\theta_2\in H^{(n)}$, $\eta_2,\theta_1\in H^{(k)}$ be vectors such that
 $\inprod{\eta_1}{\eta_1}_M = p = \inprod_M{\eta_2}{\eta_2}$ and such that $\inprod{\theta_1}{\theta_1}_M = q = \inprod_M{\theta_2}{\theta_2}$
 for projections $p,q\in M$. Let $(\chi_i)_i$ be an orthonormal basis for $pH^{(m)}_F q$. Then we compute that
 \begin{align*}
   &\inprod{(F_{n,m}\otimes \id_k)(\id_n\otimes F_{k,m})(\eta_1\otimes\eta_2\otimes\zeta)}{\xi\otimes J^{(n)}\theta_2\otimes J^{(k)}\theta_1}\\[2mm]
   &=\sum_i\begin{aligned}&\inprod{(F_{n,m}\otimes \id_k)(\eta_1\otimes\chi_i\otimes J^{(k)}\theta_1)}{\xi\otimes J^{(n)}\theta_2\otimes J^{(k)}\theta_1}\\
   &\qquad\inprod{(\id_n\otimes F_{k,m})(\eta_1\otimes\eta_2\otimes\zeta)}{\eta_1\otimes\chi_i\otimes J^{(k)}\theta_1}\end{aligned}\\[2mm]
   &=\sum_i\inprod{F_{n,m}(\eta_1\otimes\chi_i)}{\xi\otimes J^{(n)}\theta_2}
   \inprod{F_{k,m}(\eta_2\otimes\zeta)}{\chi_i\otimes J^{(k)}\theta_1}\\
   &=\sum_i \inprod{\zeta\otimes\theta_1}{F_{k,m}(J^{(k)}\eta_2\otimes\chi_i)}
   \inprod{\chi_i\otimes\theta_2}{F_{n,m}(J^{(n)}\eta_1\otimes\xi)}\\[2mm]
   &=\sum_i \begin{aligned}&\inprod{\zeta\otimes\theta_1\otimes\theta_2}{(F_{k,m}\otimes \id_n)(J^{(k)}\eta_2\otimes\chi_i\otimes\theta_2)}\\
   &\qquad\inprod{J^{(k)}\eta_2\otimes\chi_i\otimes\theta_2}{(\id_k\otimes F_{n,m})(J^{(k)}\eta_2\otimes J^{(n)}\eta_1\otimes\xi)}\end{aligned}\\[2mm]
   &=\inprod{\zeta\otimes\theta_1\otimes\theta_2}{(F_{k,m}\otimes \id_n)(\id_k\otimes F_{n,m})(J^{(k)}\eta_2\otimes J^{(n)}\eta_1\otimes\xi)}
 \end{align*}
 This implies that $F_{n+k,m}$ satisfies (\ref{eqn:FJmn}).

 \textbf{step 2:} Let $n,m\in\IN$ and let $L_1,L_2,K$ be $M$-$M$ bimodules. Then for all vectors
 $\xi\in L_1\otimes_M H^{(n)}_F\otimes_M \overline{K}$ and $\eta\in K\otimes_M H^{(m)}_F\otimes_M \overline{L_2}$
 that are bounded in all the tensor product decompositions above, we get that
 \begin{multline}
   \label{eqn:sAlg}
   T_{L_2\otimes_M H^{(n)}_F, L_1\otimes_M H^{(m)}_F}((\id_{L_1}\otimes\id_{H^{(m)}_F}\otimes J^{(n)}\otimes\id_{\overline{L_2}})(\id_{L_1}\otimes F_{n,m}\otimes\id_{\overline{L_2}})(\xi\boxdot_K\eta))\\
   =\left(T_{K\otimes_M H^{(n)}_F,L_1}\left((\id_{L_1}\otimes J^{(n)}\otimes\id_{\overline{K}})\xi\right)\otimes\id_m\right)\left(\id_K\otimes F_{m,n}\right)\left(T_{L_2,K\otimes_M H^{(m)}_F}(\eta)\otimes\id_n\right).
 \end{multline}

 The left-hand side in (\ref{eqn:sAlg}) is a closed operator while the right-hand side is bounded, so we only have to check (\ref{eqn:sAlg}) on a dense subset of
 the domain of the left-hand side. Let $\zeta_1\in H^{(n)}_F$, $\zeta_2\in L_2$, $\theta_1\in L_1$ and $\theta_2\in H^{(m)}$ be bi-bounded vectors. Then we have to show that
 \begin{align*}
   &\inprod{\theta_1\otimes\theta_2\otimes J^{(n)}\zeta_2\otimes\overline\zeta_1}{(\id_{L_1}\otimes F_{n,m}\otimes\id_{L_2})(\xi\boxdot_K\eta)}\\
   &=\inprod{T_{L_1,K\otimes_M H^{(n)}_F}\left(\overline{(\id_{L_1}\otimes J^{(n)}\otimes\id_{\overline{K}})\xi}\right)\theta_1\otimes\theta_2}{(\id_K\otimes F_{m,n})\left(T_{L_2,K\otimes_M H^{(m)}_F}(\eta)\zeta_1\otimes\zeta_2\right)}.
 \end{align*}
 For fixed bi-bounded vectors $\zeta_1,\zeta_2,\theta_1,\theta_2$, both the left hand side and the right hand side of the above expression is
 continuous in $\xi$, provided $\eta$ is bounded for both tensor product decompositions of $K\otimes_M H^{(m)}_F\otimes_M \overline{L_2}$. The same
 is true about continuity in the variable $\xi$. Hence we can assume that $\xi=\xi_1\otimes\xi_2\otimes\xi_2$ and $\eta=\eta_1\otimes\eta_2\otimes\eta_3$
 where all of $\xi_1,\xi_2,\xi_3,\eta_1,\eta_2,\eta_3$ are bi-bounded vectors.

 Now the left hand side reduces to
 \begin{align*}
   &\inprod{\theta_1\otimes\theta_2\otimes J^{(n)}\zeta_2\otimes\overline\zeta_1}{(\id_{L_1}\otimes F_{n,m}\otimes\id_{L_2})(\xi\boxdot_K\eta)}\\
   &=\inprod{\theta_1\otimes\theta_2\otimes J^{(n)}\zeta_2\otimes\overline\zeta_1}{(\id_{L_1}\otimes F_{n,m}\otimes\id_{L_2})(\xi_1\otimes\xi_2\otimes\inprod{\overline{\xi_3}}{\eta_1}_M\eta_2\otimes\eta_3)}\\
   &=\inprod{\inprod{\xi_1}{\theta_1}_M\theta_2\otimes J^{(n)}(\zeta_2)\inprod_M{\overline{\zeta_1}}{\eta_3}}{F_{n,m}(\xi_2\otimes\inprod{\overline{\xi_3}}{\eta_1}_M\eta_2)}\\
   &=\inprod{\inprod{\xi_1}{\theta_1}_M\theta_2\otimes J^{(n)}(\inprod{\overline{\eta_3}}{\zeta_1}_M\zeta_2)}{F_{n,m}(\xi_2\otimes\inprod{\overline{\xi_3}}{\eta_1}_M\eta_2)}\\
 \end{align*}

 The right hand side reduces to
 \begin{align*}
   &\inprod{T_{L_1,K\otimes_M H^{(n)}_F}(\overline{(\id_{L_1}\otimes J^{(n)}\otimes\id_{\overline{K}})\xi})\theta_1\otimes\theta_2}{(\id_K\otimes F_{m,n})(T_{L_2,K\otimes_M H^{(m)}_F}(\eta)\zeta_1\otimes\zeta_2)}\\
   &=\inprod{\overline{\xi_3}\otimes J^{(n)}\xi_2\otimes\inprod{\xi_1}{\theta_1}_M\theta_2}{(\id_K\otimes F_{m,n})(\eta_1\otimes\eta_2\otimes \inprod{\overline{\eta_3}}{\zeta_1}_M\zeta_2)}\\
   &=\inprod{J^{(n)}\xi_2\otimes\inprod{\xi_1}{\theta_1}_M\theta_2}{F_{m,n}(\inprod{\overline{\xi_3}}{\eta_1}_M\eta_2\otimes \inprod{\overline{\eta_3}}{\zeta_1}_M\zeta_2)}\\
   &=\inprod{F_{n,m}(J^{(n)}\xi_2\otimes\inprod{\xi_1}{\theta_1}_M\theta_2)}{\inprod{\overline{\xi_3}}{\eta_1}_M\eta_2\otimes \inprod{\overline{\eta_3}}{\zeta_1}_M\zeta_2}\\
 \end{align*}

 These expressions are equal by step 1.

 \textbf{step 3:} An elementary but rather tedious computation using step 2 and proposition \ref{prop:L} shows that
 $\xi\boxtimes_k \eta$ is indeed a bounded vector for all $k$, whenever $\xi\in H^{(n)}_F,\eta\in H^{(m)}_F$ are bounded vectors. Moreover we get that
 \[W(\xi)W(\eta)=\sum_{k=0}^{\min(n,m)}W(\xi\boxtimes_k \eta).\]
 \def\niets{
 Let $\xi\in H^{(n)}$ and $\eta\in H^{(m)}$. Then we can make the following computation:
 \begin{equation*}
   W(\xi)W(\eta)
   =\sum_{i=0}^n\sum_{j=0}^m L(S_{i,n-i}(\xi))L(S_{m-j,j}(\eta)).
 \end{equation*}
 By lemma \ref{lem:prodL}, this is equal to
 \begin{multline*}
   \sum_{i=0}^n\sum_{j=0}^m\sum_{k=0}^{\min(i,j)}
   L\left(\begin{aligned}
       &I_{n-i,j-k}(S_{i,n-i}(\xi)\otimes\id_{j-k})(I_{k,i-k}\otimes\id_{j-k})\\
       &\qquad\qquad(\id_k\otimes F_{i-k,j-k})\\
       &\qquad(I_{k,j-k}^\ast\otimes\id_{i-k})(S_{m-j,j}(\eta)\otimes\id_{i-k})I^{\ast}_{m-j,i-k}
     \end{aligned}\right)
 \end{multline*}
 Writing $S_{i,n-i}$ and $S_{m,m-j}$ in terms of operators $T_{K,L}$, we see that this is equal to
 \begin{equation*}
   \sum_{i=0}^n\sum_{j=0}^m\sum_{k=0}^{\min(i,j)}
   L\left(\begin{aligned}
       &I_{n-i,j-k}\left(T_{H^{(i)}_F,H^{(n-i)}_F}\left((\id_{n-i}\otimes J^{(i)})I_{n-i,i}^\ast\xi\right)\otimes\id_{j-k}\right)\\
       &\qquad\quad(I_{k,i-k}\otimes\id_{j-k})(\id_k\otimes F_{i-k,j-k})(I_{k,j-k}^\ast\otimes\id_{i-k})\\
       &\qquad\left(T_{H^{(m-j)}_F,H^{(j)}_F}\left((\id_j\otimes J^{(m-j)})I_{j,m-j}^\ast\eta\right)\otimes\id_{i-k}\right)I^{\ast}_{m-j,i-k}
       \end{aligned}\right)
 \end{equation*}
 Using the fact that $ST_{K,L}(\xi)=T_{K,\widetilde L}((S\otimes \id_K)\xi)$ for every bounded operator $S:L\rightarrow \widetilde L$,
 we see that this is equal to
 \[\sum_{i=0}^n\sum_{j=0}^m\sum_{k=0}^{\min(i,j)}
 L\left(\begin{aligned}
     &I_{n-i,j-k}\\
     &\left(T_{H^{(k)}_F\otimes_M H^{(i-k)}_F,H^{(n-i)}_F}\left((\id_{n-i}\otimes J^{(i-k)}\otimes J^{(k)})I_{n-i,i-k,k}^\ast\xi\right)\otimes\id_{j-k}\right)\\
     &(\id_k\otimes F_{i-k,j-k})\\
     &\left(T_{H^{(m-j)}_F,H^{(k)}_F\otimes_M H^{(j-k)}_F}\left((\id_j\otimes J^{(m-j)})I_{k,j-k,m-j}^\ast\eta\right)\otimes\id_{i-k}\right)\\
     &I^{\ast}_{m-j,i-k}\\
   \end{aligned}\right)
 \]
 In step 2 we showed that this is equal to
 \[
 \sum_{i=0}^n\sum_{j=0}^m\sum_{k=0}^{\min(i,j)}
 L\left(\begin{aligned}&I_{n-i,j-k}\\
     &T_{L_1, L_2}\left(\begin{aligned}
     &(\id_{n-i}\otimes\id_{j-k}\otimes J^{(i-k)}\otimes \id_{m-j})\\
     &(\id_{n-i}\otimes F_{j-k,i-k}\otimes \id_{m-j})\\
     &\left(\begin{aligned}&((\id_{n-i}\otimes \id_{i-k}\otimes J^{(k)})I_{n-i,i-k,k}^\ast\xi)\\&\boxdot_{H^{(k)}_F} ((\id_j\otimes J^{(m-j)})I_{k,j-k,m-j}^\ast\eta)\end{aligned}\right)
     \end{aligned}\right)\\
     &I^{\ast}_{m-j,i-k}\end{aligned}\right),
 \]
 where $L_1=H^{(m-j)}_F\otimes_M H^{(i-k)}_F$ and $L_2=H^{(n-i)}_F\otimes H^{(j-k)}_F$.

 Again using the fact that $ST_{K,L}(\xi)=T_{K,\widetilde L}((S\otimes \id_K)\xi)$ for every bounded operator $S:L\rightarrow \widetilde L$,
 this is the same as
 \[\sum_{i=0}^n\sum_{j=0}^m\sum_{k=0}^{\min(i,j)}
 L\left(T_{H^{(l_1)}_F, H^{(l_2)}_F}\left(\begin{aligned}
 &(\id_{l_2}\otimes J^{(l_1)})
 (I_{n-i,j-k}\otimes I_{i-k,m-j})\\
 &(\id_{n-i}\otimes F_{j-k,i-k}\otimes \id_{m-j})
 (I_{n-i,i-k}^\ast\otimes I_{j-k,m-j}^\ast)\\
 &\left(((\id_{n-k}\otimes J^{(k)})I_{n-k,k}^\ast\xi) \boxdot_{H^{(k)}_F} (I_{k,m-k}^\ast\eta)\right)
 \end{aligned}\right)\right),
 \]
 where $l_1=m-j+i-k$ and $l_2=n-i+j-k$.

 Changing the summation variables to $k$, $l=m-j+i-k$ and $r=n-i$, we obtain
 \[\sum_{k=0}^{\min(n,m)}\sum_{l=0}^{n+m-2k}\sum_{r=\max(0,n-l-k)}^{\min(l_2,n-k)}
 L\left(T_{H^{(l)}_F, H^{(l_2)}_F}\left(\begin{aligned}
       &(\id_{l_2}\otimes J^{(l)})
       (I_{r,l_2-r}\otimes I_{n-k-r,r+l+k-n})\\
       &(\id_{r}\otimes F_{l_2-r,n-r-k}\otimes \id_{r+l+k-n})\\
       &(I_{r,n-r-k}^\ast\otimes I_{l_2-r,r+l+k-n}^\ast)\\
       &\left(((\id_{n-k}\otimes J^{(k)})I_{n-k,k}^\ast\xi) \boxdot_{H^{(k)}_F} (I_{k,m-k}^\ast\eta)\right)
     \end{aligned}\right)\right),\]
 where we write $l_2=m+n-2k-l$.


 Using (\ref{eq:IIast}) from the proof of lemma \ref{lem:prodL}, this reduces to
 \[
 \sum_{k=0}^{\min(n,m)}\sum_{l=0}^{n+m-2k}
 L\left(T_{H^{(l)}_F, H^{(m+n-2k-l)}_F}\left(\begin{aligned}
       &(\id_{l}\otimes J^{(m+n-2k-l)})
       I_{l,m+n-2k-l}^\ast\\
       &I_{n-k,m-k}
       (((\id_{n-k}\otimes J^{(k)})I_{n-k,k}^\ast\xi) \boxdot_{H^{(k)}_F} (I_{k,m-k}^\ast\eta))
     \end{aligned}\right)\right)\]

 Rewriting this in terms of $S_{k,l}$ and $\boxtimes_k$, we obtain
   \[\sum_{k=0}^{\min(n,m)}\sum_{l=0}^{n+m-2k}
   L(S_{l,m+n-2k-l}(\xi\boxtimes_k\eta))=
   \sum_{k=0}^{\min(n,m)}W(\xi\boxtimes_k\eta)
 \]}
\end{proof}

\begin{definition}
  The relative Gaussian construction $\ppGauss_M(H,F,J)$ is the von Neumann algebra generated by
  the spectral projections of the operators in $\Gauss_M(H,F,J)$.
\end{definition}

It follows from the result above that $\bbH^{(n)}_F$ is dense in $H^{(n)}_F$. In particular, $\Omega$
is a cyclic vector for $\ppGauss_M(H,F,J)$. We define a state $\varphi$ on $\ppGauss_M(H,F,J)$ by
the relation that $\varphi(x)=\inprod{\Omega}{x\Omega}$.
\begin{theorem}
  The state $\varphi$ defined above is a faithful normal trace on $\ppGauss_M(H,F,J)$ that satisfies $\varphi=\tau\circ E$. In particular,
  the conditional expectation $E:\ppGauss_M(H,F,J)\rightarrow M$ is faithful.
\end{theorem}
\begin{proof}
  We know already that $\Omega$ is a cyclic vector for $\ppGauss_M(H,F,J)$. Moreover, we see that
  $\widetilde J W(\xi)\widetilde J\Omega=\widetilde J\xi=W(\widetilde J\xi)\Omega$.
  Hence it suffices to show that $W(\xi)$ commutes with $\widetilde JW(\eta)\widetilde J$.
  For each $n\in\IN$, we know that the space $\bbH^{(n)}_F$ is dense in $H^{(n)}_F$, and moreover, the
  operators $W(\xi)$ are bounded when restricted to $H^{(n)}_F$. So it suffices to show that
  \[W(\xi)\widetilde JW(\eta)\widetilde JW(\zeta)\Omega=\widetilde JW(\eta)\widetilde JW(\xi)W(\zeta)\Omega,\]
  for all bounded vectors $\xi\in H^{(n)}_F$, $\eta\in H^{(m)}_F$ and $\zeta\in H^{(k)}_F$. We compute that
  \begin{align*}
    \widetilde JW(\eta)\widetilde JW(\zeta)\Omega&=\widetilde J\sum_{i=0}^{\min(m,k)}W(\eta\boxtimes_i J^{(k)}\zeta)\Omega\\
    &=\sum_{i=1}^{\min(m,k)}W(\zeta\boxtimes_i J^{(m)}\eta)\Omega\\
    &=W(\zeta)W(J^{(m)}\eta)\Omega
  \end{align*}
  So we find that
  \begin{align*}
    W(\xi)\widetilde JW(\eta)\widetilde JW(\zeta)\Omega
    &=W(\xi)W(\zeta)W(J^{(m)}\eta)\Omega\\
    &=\sum_{i=0}^{\min(n,k)}W(\xi\boxtimes_i\zeta)W(\eta)\Omega\\
    &=\sum_{i=0}^{\min(n,k)}\widetilde JW(\eta)\widetilde JW(\xi\boxtimes_i\zeta)\Omega\\
    &=\widetilde JW(\eta)\widetilde JW(\xi)W(\zeta)\Omega.
  \end{align*}

  It follows that $W(\xi)$ commutes with $\widetilde JW(\eta)\widetilde J$, and hence that $\varphi$ is a faithful normal trace.
\end{proof}

\def\niets{
Many parts in the construction of the relative Gaussian construction depend on the choice of the faithful normal trace
on $M$, but we have never made that explicit in the notation. In fact, the von Neumann algebra $\ppGauss_M(H,F,J)$ does
not depend on this choice, and neither does the conditional expectation $E$. The state $\varphi$ however depends
on the choice of $\tau$ since it satisfies $\varphi=\tau\circ E$. In order to prove this, we
write explicitly $\cF_{(M,\tau)}(H,F)$ and $\ppGauss_{(M,\tau)}(H,F,J)$.
\begin{theorem}
  Let $\tau, \tau^\prime$ be two faithful normal traces on $M$. Let $H$ be an $M$-$M$ bimodule and let
  $F:H\otimes_M H\rightarrow H\otimes_M H$ and $J:H\rightarrow H$ be as before. Then there is a $\ast$-isomorphism
  $\psi:\ppGauss_{(M,\tau)}(H,F,J)\rightarrow\ppGauss_{(M,\tau^\prime)}(H,F,J)$ that satisfies
  $E_{\tau^\prime}\circ\psi=E_\tau$.
\end{theorem}
\begin{proof}
  Since $\tau$ is faithful, we find a vector $\xi\in\Lp^2(M,\tau)$ such that $\tau^\prime(x)=\inprod{\xi}{x\xi}$ for
  all $x\in M$. We can choose this vector $\xi$ to be $M$-central because $\tau^\prime$ is a trace.

  We define unitaries $U^{(n)}:H^{(n)}_{\tau^\prime}\rightarrow H^{(n)}_{\tau}$ by the relation that
  \[U^{(n)}\eta=l(\eta)\xi\text{ for all left-bounded vectors }\eta\in H^{(n)}.\]
  Observe that these unitaries are $M$-$M$ bimodular, and hence that they map (left or right) $\tau^\prime$-bounded vectors to
  (left or right) $\tau$-bounded vectors.
\end{proof}
}
\section{Proof of the Main Result}
\label{sect:main}
This section is devoted to the proof of the following theorem. The proof resembles very closely the proof of the main result of
\cite{HM12}, but our proof is a little less technical because the radial structure of relative Gaussian
constructions is easier to handle.
\begin{theorem}[see theorem \ref{thm:main:intro}]
  \label{thm:main}
  Let $M$ be a von Neumann algebra and let $H$ be a Hilbert $M$-$M$ bimodule. Assume that $F$ is a projection
  onto
  an $M$-$M$ subbimodule of $H\otimes_M H$ such that $F\otimes1$ commutes with $1\otimes F$ on
  $H\otimes_M\otimes_MH$.
  Let $\psi:\IN\rightarrow\IC$ be a function in class $\cC$ defined above. Then there is a unique
  ultraweakly continuous, completely bounded map
  $\Phi_\psi:\pcT_{M}(H,F)\rightarrow\Bounded_M\pcT_M(H,F)$ that satisfies
  \begin{align*}
    \Phi_\psi(L(T))
    &=\psi(n+m)L(T)
  \end{align*}
  for all bounded right-$M$-linear operators $T:H^{(n)}_F\rightarrow H^{(m)}_F$. Moreover,
  the completely bounded norm is less than
  $\norm{\Phi_\psi}_{cb}\leq\norm{\psi}_{\cC}.$
\end{theorem}

We prove the theorem by a series of lemmas. For this section, fix a tracial von Neumann algebra $(M,\tau)$,
an $M$-$M$ bimodule $H$ and an $M$-$M$ bimodular projection $F:H\otimes H\rightarrow H\otimes H$ that satisfies the relation
\[(F\otimes\id)(\id\otimes F)=(\id\otimes F)(F\otimes\id).\]

The first lemma provides a way to estimate the
completely bounded norm of an operator of the form $\Phi(T)=\sum_i u_iTv_i$. It was proven by
Christensen and Sinclair in \cite{CS}.
\begin{lemma}[{see \cite[Corollary 6.2]{CS}}]
  \label{lem:cb}
  Let $H,K$ be Hilbert spaces and $(u_i),(v_i)$ sequences of bounded operators from $H$ to $K$.
  If $\sum_i u_iu_i^\ast$ and $\sum_i v_i^\ast v_i$ are bounded operators, then the formula
  $\Phi(T)=\sum_i u_iTv_i^\ast$ defines a completely bounded map $\Phi:\Bounded(H)\rightarrow\Bounded(K)$.
  The completely bounded norm of $\Phi$ is bounded by
  $\norm{\Phi}_{cb}^2\leq \norm{\sum_i u_iu_i^\ast}\norm{\sum_i v_i v_i^\ast}$.
\end{lemma}
\def\niets{
\begin{proof}
  Observe that, for every $n\in \IN$, the map $\Phi\otimes\id_{\MatM_n(\IC)}$ is given by
  \[\Phi\otimes\id_{\MatM_n(\IC)}(T)=\sum_i (u_i\otimes 1)T(v_i\otimes 1)^\ast.\]
  So we only have to prove that $\Phi$ defines a bounded map from $\Bounded(H)$ to $\Bounded(K)$,
  and that its norm satisfies the relation above. I.e.\@ we have to show that
  \[\abs{\inprod{\xi}{\Phi(T)\eta}}
  \leq \norm{\sum_i u_iu_i^\ast}^{\frac12}\norm{\sum_i v_i v_i^\ast}^{\frac12}\norm{T}\norm{\xi}\norm{\eta}.\]
  This follows from the following computation
  \begin{align*}
    \abs{\inprod{\xi}{\Phi(T)\eta}}
    &\leq\sum_i\inprod{u_i^\ast\xi}{Tv_i^\ast\eta}\\
    &\leq\sum_i\norm{u_i^\ast\xi}\norm{T}\norm{v_i^\ast\eta}\\
    &\leq\norm{T}\left(\sum_i\norm{u_i^\ast\xi}^2\right)^{\frac12}
    \left(\sum_i\norm{v_i^\ast\xi}^2\right)^{\frac12}\\
    &\leq\norm{T}\inprod{\xi}{\left(\sum_i u_iu_i^\ast\right)\xi}^{\frac12}
    \inprod{\eta}{\left(\sum_i v_iv_i^\ast\right)\xi}^{\frac12}\\
    &\leq\norm{T}\norm{\xi}\norm{\eta}\norm{\sum_i u_iu_i^\ast}^{\frac12}\norm{\sum_i v_i v_i^\ast}^{\frac12}
  \end{align*}
\end{proof}
}

Whenever $x\in\ell^\infty(\IN)$,
we define a radial multiplication operator $M_x$ on $\cF_{M}(H,F)$ by the relation that
$M_x\eta=x(n)\eta$ whenever $\eta\in H^{(n)}_{F}$.
On $\ell^\infty(\IN)$, we consider the one-directional shift $S$, which is defined by $S(x)_n=x_{n-1}$ when $n>0$ and $S(x)_0=0$.
We denote the shift in the other direction by $S^\ast$.

\begin{remark}
\label{rem:M}
It is clear that the radial multiplication operators are $M$-$M$ bimodular and they
satisfy the following relations with the creation operators.
\begin{equation*}
  M_xL(T)M_y=M_{x((S^\ast)^nS^m y)}L(T)=L(T)M_{y (S^{n}(S^{\ast})^m x)}
\end{equation*}
for all bounded right-$M$ modular operators $T:H^{(n)}_F\rightarrow H^{(m)}_F$.
\end{remark}

We consider sequences $z,r_k\in\ell^\infty(\IN)$ that are defined by $z_n=(-1)^{n}$ and $r_k(n)=1$ whenever $n\geq k$ and $r_k(n)=0$ otherwise.
Observe that $u=M_z$ is a unitary and that the $q_k=M_{r_k}$ are projections.

As a left $M$-module, we can write $H$ in the form
$\bimod{M}{H}{}\cong\bigoplus_i \bimod{M}{\Lp^2(M,\varphi)p_i}{}$. We write $\xi_i$ for the vector corresponding
to $\hat1p_i$ in the $i$-th component of the direct sum above.
Observe that $\inprod_M{\xi_i}{\xi_j}=\delta_{i,j}p_i$.
\begin{lemma}
  \label{lem:rho}
  Define a completely positive map $\rho:\Bounded(\cF_{M,\varphi}(H,F))\rightarrow \Bounded(\cF_{M,\varphi}(H,F))$
  by the formula $\rho(T)=\sum_i R(\xi_i)TR(\xi_i)^\ast$.
  This operator satisfies
  \[\rho^l(L(T))=L(I_{m,l}(T\otimes\id_l)I^\ast_{n,l})=L(T)q_{n+l},\]
  for all bounded right-$M$ modular operators $T:H^{(n)}_F\rightarrow H^{(m)}_F$ and $l\geq 0$.
  In particular, $\rho$ is a subunital completely positive map, because $1=L(\id_0)$ and $q_1=L(\id_1)$ is a projection.
\end{lemma}
\begin{proof}
  We first prove lemma \ref{lem:rho} with $l=1$.
  Let $T:H^{(n)}_F\rightarrow H^{(m)}_F$ be a bounded right-$M$-modular operator. It is clear that $\rho(L(T))\eta=0=L(I_{m,1}(T\otimes\id)I_{n,1}^\ast)\eta$
  whenever $\eta\in H^{k}$ for some $k < n+1$. Let $\eta\in H^{n+k+1}$ for some $k\geq 0$.
  We compute that
  \begin{align*}
    \rho(L(T))\eta&=\sum_i I_{m+k,1}(\id_{m+k}\otimes r(\xi_i))I_{m,k}(T\otimes\id_1)I_{n,k}^\ast(\id_{n+k}\otimes r(\xi_i)^\ast)I_{n+k,2}^\ast\eta\\
    &\sum_i I_{m,k,1}(T\otimes \id_{k}\otimes r(\xi_i)r(\xi_i)^\ast)I_{n,k,1}^\ast\eta\\
    &=I_{m,k,1}(T\otimes\id_k\otimes\id_1)I_{n,k,1}^\ast\\
    &=I_{m,k+1}(T\otimes (I_{k,1}I_{k,1}^\ast))I_{n,k+1}^\ast\\
    \text{and }L(T\otimes\id)\eta&=I_{m,1,k}(T\otimes\id_1\otimes\id_k)I_{n,1,k}^{\ast}\\
    &=I_{m,k+1}(T\otimes(I_{1,k}I_{1,k}^\ast))I_{n,k+1}^{\ast}\\
    \text{and }L(T)q_{n+1}\eta&=I_{m,k+1}(T\otimes\id_{k+1})I_{m,k+1}^\ast.
  \end{align*}

  When $F$ is a projection such that $F\otimes\id_1$ commutes with $\id_1\otimes F$, then we get that $D^{(n)}:H^{(n)}\rightarrow H^{(n)}$
  is simply the projection onto the orthogonal complement of the closed linear span of the subspaces of the form $H^{(i-1)}\otimes_M F(H^{(2)})\otimes_M H^{(n-i-1)}$
  for $i=1,\ldots,n-1$. So $H^{(n)}_F$ is a closed subspace of $H^{(n)}$. The operator $I_{n,m}$ is simply the projection
  from $H^{(n)}_F\otimes_M H^{(m)}_F$ onto its closed subspace $H^{(n+m)}_F$. Hence we see that $I_{n,m}I_{n,m}^{\ast}=\id_{n+m}$, so
  it follows that $L(T\otimes\id)\eta=L(T)q_{n+1}\eta$. This is true for all $\eta\in H^{k}$, for all $k\in\IN$, so $\rho(L(T))=L(T\otimes\id)=L(T)q_{n+1}$.

  Now let $l>1$. By induction we get that
  \begin{align*}
    \rho^l(L(T))&=L(I_{m,1,\ldots,1}(T\otimes\id_1\otimes\ldots\otimes\id_1)I_{n,1,\ldots,1})\\
    &=L(I_{m,l}(T\otimes I_{1,\ldots,1}I_{1,\ldots,1}^\ast)I_{n,l}^\ast.
  \end{align*}
  By the same argument as above, we see that $_{1,\ldots,1}I_{1,\ldots,1}^\ast=\id_l$, so indeed
  \[\rho^l(L(T))=L(I_{m,l}(T\otimes\id_l)I_{n,l}^\ast)=L(T)q_{n+l}.\]
\end{proof}

\begin{lemma}
  \label{lem:Phi}
  Let $x,y\in\ell^2(\IN)$, and define an $M$-$M$-bimodular map $\Phi_{x,y}:\pcT_M(H,F)\rightarrow \pcT_M(H,F)$
  by the formula
  \[\Phi_{x,y}(T)=\sum_{n\geq 0}M_{(S^\ast)^nx}TM_{(S^\ast)^ny}^\ast+\sum_{n\geq 1}M_{S^nx}\rho^n(T)M_{S^ny}^\ast.\]
  Then this operator is completely bounded with completely bounded norm $\norm{\Phi_{x,y}}_{cb}\leq \norm{x}_2\norm{y}_2$
  that satisfies
  \begin{equation}
    \label{eq:phi}
    \Phi_{x,y}(L(T))=\inprod{(S^\ast)^mx}{(S^\ast)^ny}L(T).
  \end{equation}
\end{lemma}
\begin{proof}
  Observe that $\Phi_{x,y}$ is given in the form of lemma \ref{lem:cb}, and that in the notation of lemma \ref{lem:cb}, we have that $\sum_{i}u_iu_i^\ast=\Phi_{x,x}(1)$
  and $\sum_i v_i v_i^\ast=\Phi_{y,y}(1)$. So, once we show (\ref{eq:phi}), it follows from lemma \ref{lem:cb} that $\Phi_{x,y}$ is completely bounded and
  its norm is bounded by $\norm{\Phi_{x,y}}_{cb}\leq \norm{x}_2\norm{y}_2$.

  Let $T:H^{(n)}_F\rightarrow H^{(m)}_F$ be a bounded right-$M$-linear operator. Then we see that
  \begin{align*}
    \Phi_{x,y}(L(T))&=\sum_{k\geq 0}M_{(S^\ast)^kx}L(T)M_{(S^\ast)^ky}^\ast+\sum_{k\geq 1}M_{S^kx}\rho^k(L(T))M_{S^ky}^\ast\\
    &=L(T)\left(\sum_{k\geq 0}M_{S^n(S^\ast)^{k+m}x}M_{(S^\ast)^ky}^\ast + \sum_{k\geq 1}M_{S^{n+k}(S^{\ast})^{m+k}S^kx}M_{S^ky}^\ast q_{n+k}\right)\\
    &=L(T)M_{f},
  \end{align*}
  where the function $f$ is given by
  \begin{align*}
    f(l)&=0&&\text{ if }l < n\\
    f(n+l)&=\sum_{k\geq 0} x(l+m+k)\overline{y(l+n+k)}+\sum_{k=1}^{l} x(l-k+m)\overline{y(l+n-k)}&&\text{ if } l\geq 0\\
    &=\sum_{k\geq l} x(k+m)\overline{y(k+n)} + \sum_{k=0}^{l-1} x(k+m)\overline{y(k+n)}\\
    &=\inprod{(S^\ast)^n y}{(S^{\ast})^m x}.
  \end{align*}

  Since $L(T)=L(T)q_n$, we see that indeed
  \[\Phi_{x,y}(L(T))=L(T)\inprod{(S^\ast)^n y}{(S^{\ast})^m x}.\]
\end{proof}

\begin{proof}[proof of theorem \ref{thm:main}]
It is well-known that every trace class operator $T\in\Bounded(\ell^2\IN)$ can be written as a sum
of rank one operators $T=\sum_n x_n\otimes y_n$ with $x_n,y_n\in \ell^2(\IN)$ and where the sum converges in
trace norm. Moreover, the trace norm of $T$ is $\sum_n\norm{x_n}_2\norm{y_n}_2$.
In particular, for the H\"ankel matrix $H_\psi$, we find sequences of vectors $x_n,y_n\in\ell^2(\IN)$
such that
\[\psi(k+l)-\psi(k+l+2)=H_\psi(k,l)=\sum_n x_n(k)\overline{y_n(l)}\text{ for all }k,l\in\IN,\]
and such that $\norm{H_\psi}_1=\sum_n\norm{x_n}_2\norm{y_n}_2$.
Moreover, we see that
\[\psi(k)=c_+ + (-1)^kc_- + \sum_m (\psi(k+2m)-\psi(k+2m+2)).\]
So we also see that
\begin{align*}
  \psi(k+l)&=
  c_+ + (-1)^{k+l}c_- + \sum_{m,n}x_n(k+m)\overline{y_n(l+m)}\\
  &=c_+ + (-1)^{k+l}c_- + \sum_{n}\inprod{(S^{\ast})^ly_n}{(S^{\ast})^kx}
\end{align*}

Now, we set $\Phi_\psi=c_+ + c_- \Ad_u + \sum_k \Phi_{x_k,y_k}$, where $u=M_z$ with $z(k)=(-1)^k$ as before.
Let $T:H^{(n)}_F\rightarrow H^{(m)}_F$ be a bounded right-$M$-linear operator.
Then it follows from lemma \ref{lem:Phi} that
\begin{align*}
  \Phi_\psi(L(T))
  &=c_+ + c_- L(T)M_{S^{m}(S^\ast)^nz}M_z + \sum_k \Phi_{x_k,y_k}(L(T))\\
  &=\left(c_+ + c_- (-1)^{n+m} + \sum_n \inprod{(S^\ast)^my}{(S^\ast)^nx}\right)L(T)\\
  &=\psi(n+m)L(T)\\
\end{align*}
So $\Phi_\psi$ is indeed the map we were searching for. Moreover $\Phi_\psi$ is completely bounded and its completely bounded norm is bounded by
\[\norm{\Phi_\psi}_{cb}\leq \abs{c_+} + \abs{c_-} + \sum_{k} \norm{x_k}_2\norm{y_k}_2=\norm{\psi}_{\cC}.\]
\end{proof}

\section{Amalgamated free product von Neumann algebras}
\label{sect:cons}
In this section we deduce the following theorem from theorem \ref{thm:main}.
\begin{theorem}
  \label{thm:main:prod}
  Let $(M_i)_i$ be a (finite or countably infinite) family of tracial von Neumann algebras with a common subalgebra $P$.
  Denote by $M$ the amalgamated free product over $P$
  of the von Neumann algebras $M_i$. Let $\psi:\IN\rightarrow \IC$ be a function in class $\cC^\prime$ defined in the introduction.
  Then there is a unique ultraweakly continuous completely bounded map $\Psi_\psi:M\rightarrow M$ that satisfies
  \[\Psi_\psi(x_1\ldots x_n)=\psi(n)x_1\ldots x_n\text{ for every reduced word }x_1\ldots x_n\in M.\]
  Moreover, the completely bounded norm of $\Psi_\psi$ is bounded above by $\norm{\Psi_\psi}_{cb}\leq\norm{\psi}_{\cC^\prime}$.
\end{theorem}
\begin{proof}
  Consider the graph of von Neumann algebras that is depicted below
  \\\hspace*{\stretch{1}}
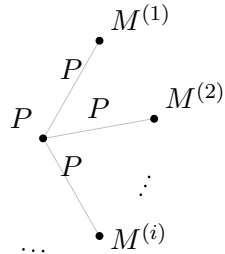
\begin{tikzpicture}[thick]
    \fill (0,0) circle(1.5pt) node[above left]{$P$};
    \fill (0,0) -- (60:1.5cm) node[pos=0.5, above]{$P$} circle(1.5pt) node[above right]{$M^{(1)}$};
    \fill (0,0) -- (10:1.5cm) node[pos=0.5, above]{$P$} circle(1.5pt) node[above right]{$M^{(2)}$};
    \fill (-20:1.5cm) circle(0.5pt) (-25:1.5cm) circle(0.5pt) (-30:1.5cm) circle(0.5pt);
    \fill (0,0) -- (-60:1.5cm) node[pos=0.5, above]{$P$} circle(1.5pt) node[right]{$M^{(i)}$};
    \fill (-90:1.5cm) circle(0.5pt) (-95:1.5cm) circle(0.5pt) (-100:1.5cm) circle(0.5pt);
  \end{tikzpicture}\hspace*{\stretch{1}}\\

  As in the introduction, we can consider the fundamental von Neumann algebra $\widetilde M$ of this graph.
  In other words, we set $\widetilde M=\Gamma^{\prime\prime}_{\widetilde P}(H,F,J)$ where
  \begin{align*}
    \widetilde P&=P\oplus\bigoplus_i M_i\\
    H&=\bigoplus_i \bimod{M_i}{\Lp^2(M_i)}{P} \oplus \bimod{P}{\Lp^2(M_i)}{M_i}\\
    J(\hat x)&=\widehat{x^\ast}\in \bimod{P}{\Lp^2(M_i)}{M_i}\qquad\text{ for all }x\in\bimod{M_i}{\Lp^2(M_i)}{P}\text{ and vice-versa}.
  \end{align*}
  and where $F$ is the projection onto the $\widetilde P$-$\widetilde P$ Hilbert subbimodule
  \begin{align*}
    L&=\bigoplus_i \bimod{P}{P}{P}\oplus\bigoplus_i \bimod{M^{(j)}}{\overline{M^{(i)}\otimes_P M^{(i)}}}{M^{(i)}}\\
    &\subset \bigoplus_i \bimod{P}{\overline{M^{(i)}}}{P} \oplus \bigoplus_{i,j} \bimod{M^{(i)}}{\overline{M^{(i)}\otimes_P M^{(j)}}}{M^{(j)}}\\
    &=H^{(2)}.
  \end{align*}

  Denote by $p$ the central projection in $\widetilde P\subset \widetilde M$ that corresponds to the term $P$ in the direct sum $\widetilde P=P\oplus\bigoplus_i M^{(i)}$.
  Then it follows that
  \begin{equation}
    \label{eq:prod}
    p\cF_{\widetilde P}(H,F)p\cong\bimod{P}{\Lp^2(P)}{P}\oplus \bigoplus_i \bimod{P}{\Lp^2(M_i\ominus P)}{P} \oplus
    \bigoplus_{i\not=j} \bimod{P}{\Lp^2(M_i\ominus P)\otimes_P\Lp^2(M_j\ominus P)}{P}\oplus\ldots,
  \end{equation}
  a priori just as $P$-$P$ bimodules.
  Observe that each term $\Lp^2(M_i)$ that appears in the direct sum above is a direct summand
  of $H^{(2)}_{F}$ rather than of $H^{(1)}_{F}$.
  The $P$-$P$ bimodule on the right is precisely $\Lp^2(M)$. We denote by $U:p\cF_{\widetilde P,\varphi}(H,F)p\rightarrow \Lp^2(M)$
  the unitary that implements the identification in (\ref{eq:prod}).
  It is now easy to see that
  \[upW(\hat y)pu^\ast = y\in M,\]
  for every $y\in M_i\ominus P$.

  In particular, we get that $\alpha=\Ad_{u^\ast}:M\rightarrow p\widetilde Mp$ is an isomorphism that maps a reduced word $x_1\ldots x_n$ of length $n$
  to the reduced word $pW(x_1\otimes \ldots\otimes x_n)p$ that has length $2n$. Let $\psi:\IN\rightarrow \IC$ be a function in class $\cC^\prime$.
  We consider the function $\tilde \psi:\IN\rightarrow \IC$ that is defined by $\tilde\psi(2n)=\psi(n)$ and $\tilde\psi(2n+1)=0$ for all $n\in\IN$.
  Then we know that $\tilde\psi\in\cC$ and its norm is $\norm{\tilde \psi}_{\cC}=\norm{\psi}_{\cC^\prime}$. By theorem \ref{thm:main}, we find an
  ultraweakly continuous map
  $\Phi_{\tilde \psi}:\pcT_{M}(H,F)\rightarrow\pcT_M(H,F))$ that satisfies
  \begin{align*}
    \Phi_{\tilde\psi}(L(T))
    &=\psi(n+m)L(T)\\
  \end{align*}
  for all bounded right-$M$ modular operators $T:H^{(n)}_F\rightarrow H^{(m)}_F$. Moreover, the completely bounded norm of $\Phi_{\tilde\psi}$ is bounded by
  $\norm{\Phi_{\tilde\psi}}_{cb}\leq\norm{\psi}_{\cC}$.

  In particular, we get that $\Phi_{\tilde\psi}(W(\xi))=\tilde\psi(n)W(\xi)$ for all bounded vectors $\xi\in H^{(n)}_F$.
  Moreover, $\Phi_{\tilde\psi}$ is $M$-$M$-bimodular, and in particular, $\Phi_{\tilde\psi}(p\widetilde Mp)\subset p\widetilde Mp$.
  We define $\Phi_\psi:M\rightarrow M$ by the formula $\Phi_\psi(x)=\alpha^{-1}(\Phi_{\tilde\psi}(\alpha(x)))$. We conclude that
  \[\Phi_\psi(x_1\ldots x_n)=\tilde\psi(2n)x_1\ldots x_n=\psi(n)x_1\ldots x_n\]
  for all reduced words $x_1\ldots x_n$ in the amalgamated free product decomposition $M=M_1\free_PM_2\free_P\ldots$.
  Moreover, the completely bounded norm of $\Phi_\psi$ is bounded by
  \[\norm{\Phi_\psi}_{cb}\leq\norm{\Phi_{\tilde\psi}}_{cb}\leq \norm{\tilde\psi}_{\cC}=\norm{\psi}_{\cC^\prime}.\]
  So $\Phi_\psi$ satisfies the conditions of theorem \ref{thm:main}
\end{proof}


\newcommand{\etalchar}[1]{$^{#1}$}
\bibliography{references}{}
\bibliographystyle{sdpabbrv}

\end{document}